\documentclass[a4paper,10pt]{scrartcl}

\usepackage[utf8]{inputenc}
\usepackage{cite}
\usepackage{amsmath}
\usepackage{amsfonts}
\usepackage{amssymb}
\usepackage{color}
\usepackage{subfigure}
\usepackage{enumerate}
\usepackage{amsthm} 
\usepackage{thmtools} 	
\usepackage{thm-restate}
\usepackage{hyperref}
\usepackage[capitalize,nameinlink]{cleveref}
\usepackage[colorinlistoftodos]{todonotes}
\usepackage{epigraph}
\usepackage{fullpage}
\usepackage{lmodern}

\usepackage{graphicx}
\usepackage{tikz}
\usepackage{pgf}
\usepackage{pgfplots}
\pgfplotsset{compat=1.9}
\usetikzlibrary{calc}
\usetikzlibrary{positioning,arrows,arrows.meta}

\newtheorem{thm}{Theorem}[section]
\newtheorem{lem}[thm]{Lemma}
\newtheorem{propn}[thm]{Proposition}
\newtheorem{cor}[thm]{Corollary}
\newtheorem{defn}[thm]{Definition}
\newtheorem{eg}[thm]{Example}
\newtheorem{rem}[thm]{Remark}
\renewcommand{\equiv}{:=}
\newcommand{\Ball}{{\mathbb{B}}}

\newcommand{\Euclid}{\mathcal{E}}
\newcommand{\NN}{\mathbb{N}}
\newcommand{\Rbb}{\mathbb{R}}

\newcommand{\abar}{{\overline{a}}}
\newcommand{\bbar}{{\overline{b}}}
\newcommand{\ubar}{{\overline{u}}}
\newcommand{\xbar}{{\overline{x}}}
\newcommand{\xtilde}{{\widetilde{x}}}
\newcommand{\ybar}{{\overline{y}}}
\newcommand{\ytilde}{{\widetilde{y}}}
\newcommand{\epsilonbar}{{\overline{\epsilon}}}
\newcommand{\omegabar}{{\overline{\omega}}}

\newcommand{\zetabar}{{\bar{\zeta}}}
\newcommand{\epsilontilde}{{\widetilde{\epsilon}}}
\newcommand{\Zcal}{\mathcal{Z}}
\newcommand{\Mcal}{\mathcal{M}}
\newcommand{\Ycal}{{\mathcal{Y}}}
\newcommand{\paren}[1]{\left(#1\right)}
\newcommand{\klam}[1]{\left\{#1\right\}}
\newcommand{\set}[2]{\left\{#1\,\left|\,#2\right.\right\}}
\newcommand{\ip}[2]{\left\langle #1,#2\right\rangle}
\newcommand{\norm}[1]{\left\|#1\right\|}
\newcommand{\mmap}[3]{#1:\,#2\rightrightarrows #3\,}
\DeclareMathOperator{\Id}{Id}
\DeclareMathOperator{\id}{Id}
\newcommand {\Limsup} {\mathop{{\textrm{ Lim\,sup}}\,}}

\DeclareMathOperator{\dist}{dist}
\DeclareMathOperator{\cone}{{cone}}
\DeclareMathOperator{\gph}{gph}

\DeclareMathOperator{\Fix}{\mathsf{Fix}\,}
\newcommand{\ncone}[1]{{N}_{#1}}
\newcommand{\pncone}[1]{N^{{\mathrm{ P}}}_{#1}} 

\newcommand{\raar}{{T_{\mathrm{DR}\lambda}}}
\newcommand{\DR}{{T_{\mathrm{DR}}}}

\title{Convergence Analysis of the Relaxed Douglas-Rachford Algorithm}
\author{D. Russell Luke  
	\thanks{Institut für Numerische und Angewandte Mathematik, Universität Göttingen, Lotzestr. 16-18, 37083 Göttingen, Germany, Email: \href{mailto:r.luke@math.uni-goettingen.de}{r.luke@math.uni-goettingen.de}. The research of DRL was supported in part by the German Research Foundation grant SFB755-C02 and in part by the German Research Foundation grant GRK2088-B5.}
\and Anna-Lena Martins	
	\thanks{Institut für Numerische und Angewandte Mathematik, Universität Göttingen, Lotzestr. 16-18, 37083 Göttingen, Germany, Email: \href{mailto:a.martins@math.uni-goettingen.de}{a.martins@math.uni-goettingen.de}. The research of A.-L.M. was supported by the German Research Foundation grant SFB755-C02.}}
\date{\today}

\begin{document}

\maketitle
\begin{abstract}
	Motivated by nonconvex, inconsistent feasibility problems in imaging, the relaxed alternating averaged reflections
	algorithm, or relaxed Douglas-Rachford algorithm (DR$\lambda$), was first proposed over a decade ago. Convergence
	results for this algorithm are limited either to convex feasibility or consistent nonconvex feasibility with
	strong assumptions on the regularity of the underlying sets.  Using an
	analytical framework depending only on metric subregularity and pointwise almost averagedness,
	we analyze the convergence behavior of DR$\lambda$ for feasibility problems that are both nonconvex and inconsistent.
	We introduce a new type of regularity of sets, called super-regular at a distance, to establish
	sufficient conditions for local linear convergence of the corresponding sequence. These
	results subsume and extend existing results for this algorithm.
\end{abstract}


%
%


\section{Introduction}

The {\em feasibility problem} consists of finding a common point in a collection of 
closed sets. If no such common point exists, the feasibility problem 
is called {\em inconsistent} and one seeks instead an
adequate approximation to the problem. 
Typically feasibility problems are solved by projection based algorithms. Among these
are von Neumann's alternating projections  \cite{Neumann50}, 
and its many set version, the
cyclic projection algorithm, or averaged projections and, 
in the case of two-set feasibility, the Douglas-Rachford 
algorithm \cite{Douglas} as formulated by Lions and Mercier \cite{LionMerc}. 

Alternating and cyclic projections have long been standard iterative procedures.
They are stable and reliable in the sense that they always seem to converge to 
a limit cycle, though the cycle is not 
always desirable or easy to interpret \cite{BailComComi2012}. 
Because it has so many different formulations, the Douglas-Rachford algorithm 
has been rediscovered many times and has become quite popular in the 
last decade.  This algorithm has many curious features.  The first 
of which is that the iterates do not, in general, converge to solutions to 
the target feasibility problem, when they converge at all.  The second 
unusual feature of the algorithm is that it cannot converge if the feasibility
problem is inconsistent.  For convex feasibility the iterates {\em diverge} 
in the direction of the gap between the sets \cite{LionMerc, Eckstein1992, BauComLuk2004}.  
In the convex setting this is not too worrisome, since 
the {\em shadows} of the iterates, defined as the projection of the iterates
onto one of the sets (the ``inner set''), converge to a {\em best approximation
point} \cite{BauComLuk2004, BauMou16b}.  For  consistent nonconvex 
feasibility, Hesse and Luke \cite{HesseLuk2013} were the first to prove meaningful
local convergence results for Douglas-Rachford.  This was quickly followed by several 
generalizations \cite{BauschkeNoll14, Phan16, LiPong16, LukNguTam17, min18}.  
For inconsistent feasibility, since Douglas-Rachford cannot converge, weak convergence
follows generically if the iterates are bounded, but otherwise
meaningful results appear to only be possible for relaxations of the Douglas-Rachford 
algorithm.  

To address failure of convergence of Douglas-Rachford for inconsistent feasibility, 
Luke introduced the relaxed Douglas-Rachford algorithm  
in \cite{Luke2005} with a proof of convergence for convex feasibility -- 
inconsistent and consistent.  Given $x^0 \in \Euclid$ and $\lambda\in (0,1)$, for $k=0,1,2,\dots,$ 
the iteration takes the form 
\begin{equation}\label{raar}
x^{k+1}\in \raar(x^k)\equiv \set{\frac{\lambda}{2}\paren{R_A(2b-x^k) +x^k}+\paren{1-\lambda}b}{ \ b \in P_B(x^k)}.
\end{equation}
Here $R_A$ is the {\em reflector} across the set $A$ and $P_B$ is the {\em projector} onto $B$
(see the next subsection for details).  
For $\lambda=1$ this mapping is the Douglas-Rachford fixed point mapping, 
$\DR \equiv \frac{1}{2}\paren{R_AR_B +\Id}$, where $\Id $ denotes the identity.  
From here on, we will refer to the algorithm as DR$\lambda$.
A characterization of the fixed points in the nonconvex 
inconsistent case and a first attempt at a local convergence result was given in  \cite{Luke2008}.  
The analysis required one of the sets to be convex 
and the other set to be {\em prox-regular}.   More recently, Li and Pong \cite{LiPong16} 
rediscovered this algorithm and showed 
convergence results when both sets are closed, one set is convex, at least one of the sets is 
compact and the intersection is nonempty (i.e. consistent feasibility).  
When, in addition, both sets are semi-algebraic, they showed global convergence, 
\cite[Corollary 1]{LiPong16}.  Under 
still stronger assumptions (that one of the sets is linear, the other semi-algebraic, and that the intersection is 
{\em strongly regular \cite{LewLukMal2009}}), local linear convergence can be shown, 
\cite[Proposition 2]{LiPong16}.  Minh and Phan \cite[Theorem 5.8]{min18} show local linear 
convergence of a generalized Douglas-Rachford algorithm (which includes $\raar$ as a special case)
when the two sets are super-regular, and have sufficiently regular nonempty intersection.  
Noteworthy here is that none of these approaches can explain convergence in the case of two affine halfspaces with empty intersection, much less for any other inconsistent feasibility problem, convex or otherwise.  

In the present work, we  extend the results above  
to inconsistent feasibility for sets with the weakest regularity assumptions to date. 
We introduce in Section \ref{s:srat} a new kind of set-regularity, called {\em super-regularity at a 
distance} that will be our only assumption on the sets themselves.  
Super-regularity at a distance falls into the spectrum of other regularity notions 
like $\epsilon$-subregularity (cf. \cite{Kruger2018,DanLukTam19}) and, as the name suggests, 
super-regularity \cite{LewLukMal2009}.  The innovation of this characterization 
is that it allows one to describe the regularity of a set relative to a point not in 
that set; that is,   it characterizes how the set looks from the outside.  
This is especially 
important for the relaxed Douglas-Rachford algorithm, whose fixed points 
do not usually lie in any of the sets.  As in \cite{Luke2008}, however, 
the projections of the fixed points are shown to include best approximation points (Theorem \ref{thm:fixed points}
and Corollary \ref{cor: fixed point and its gap}).

Following the framework established in \cite{LukNguTam17}, in Section \ref{s: conv analysis} 
we prove local linear convergence of the algorithm under additional assumptions on 
the regularity of the {\em collection of sets} taken together.  Unlike 
previous notions of regularity of collections of sets \cite{Kruger2018}, the 
sets in the present analysis need not have points in common.  
The analysis of \cite{LukNguTam17} uses two properties of 
fixed point mappings.  The first property, 
{\em pointwise almost averagedness}, follows from the regularity 
of the sets and, as shown in \cite[Proposition 4]{LukTebTha18} is 
an important ingredient in guaranteeing convergence of the 
iterates to fixed points.   In Theorem \ref{thm: singlevalued at fixed points}  
we establish that the $\raar$ mapping is almost averaged at its fixed points 
when the sets $A$ and $B$ are super-regular at a distance.  
The second property, {\em metric subregularity} of the 
fixed point mapping at its fixed points, was subsequently shown in 
\cite[Theorem 2]{LukTebTha18} to be necessary for local linear 
convergence.   In the context of feasibility, this property becomes 
{\em subtransversality} of the sets in relation to each other, plus 
an additional technical condition.  Under these conditions 
\cite[Theorem 3.2]{LukNguTam17} establishes local linear convergence of 
cyclic projections onto 
sufficiently regular sets that need not have points in common.  
Following their approach we show in \cref{thm:conv raar} that a similar result is true 
for DR$\lambda$. 
We conclude our study with a demonstration of this theory in \cref{sec:examples} 
via several elementary examples that allow 
explicit evaluation of the relevant constants.

\subsection{Notation and Definitions}\label{subsec:notation}
Our notation is standard in variational analysis.  
Our setting is a finite dimensional Euclidean space, denoted $\Euclid$, with 
inner product $\ip{\cdot}{\cdot}$ and induced norm $\norm{\cdot}$. 
We denote by $\Ball$ the open unit ball, and by $\Ball_{\delta}(x)$ the open ball with radius $\delta$
around the point $x$. 
The model we consider is 
a feasibility problem, that is, the problem of finding points common to closed subsets of $\Euclid$, or reasonable
substitutions thereof when the sets have no points in common.  
The distance of a point $x$ to a set $C$ is  $\dist(x,C)\equiv \inf_{y\in C}\|x-y\|$ and the {\em projector} onto $C$ is the 
{\em set-valued mapping} $P_C(x)\equiv\set{z}{\|z-x\|=\dist(x,C)}$.  A {\em projection} is a selection from $P_C(x)$. The 
{\em reflector} of a point $x$ across $C$ is $R_C(x)\equiv 2P_C(x)-x$, and a {\em reflection} is a selection from this 
set-valued mapping.  For the purposes of this paper, we define the normal cone to the set $C$ in terms of the 
projector onto that set.
\begin{defn}[normal cones]
  	Let $C\subseteq\Euclid$ and let $\bar{x}\in C$.
	\begin{enumerate}[(i)]
		\item The \emph{proximal normal cone} of $C$ at $\bar{x}$ is defined by
		 	$$ N^{\mathrm P}_C(\bar{x}) = \cone\left(P^{-1}_C\bar{x}-\bar{x}\right).$$
        	Equivalently, $\bar{x}^*\in N^{\mathrm P}_C(\bar{x})$ whenever there exists $\sigma\geq 0$ such that
			$$\langle \bar{x}^*,x-\bar{x}\rangle\leq \sigma\|x-\bar{x}\|^2 \quad(\forall x\in C).$$
		
		\item The \emph{limiting (proximal) normal cone} of $C$ at $\bar{x}$ is defined by
			$$ \ncone{C}(\bar{x}) = \Limsup_{x\to\bar{x}}\pncone{C}(x), $$
			where the limit superior is taken in the sense of \emph{Painlev\'e--Kuratowski outer limit}.
	\end{enumerate}
 	 When $\bar{x}\not\in C$ all normal cones at $\bar{x}$ are empty (by definition).
\end{defn}

\section{Super-regularity at a Distance}\label{s:srat}
We limit our attention in this study to {\em super-regular} sets and their extension to sets with 
the corresponding properties relative to points {\em not} belonging to the sets. 
\begin{defn}[{super-regularity {\cite[Definition~4.3]{LewLukMal2009}}}]\label{Def-super-reg}
Let $\Omega\subseteq\mathbb{R}^{n}$ and $\xbar\in \Omega$. The set $\Omega$
is said to be \emph{super-regular at $\xbar$}
if it is locally closed at
$\xbar$ and for every $\epsilon>0$ there is a $\delta>0$   
such that for all $(x, 0)\in\gph\ncone{\Omega}\cap \left\{\left(\Ball_\delta(\xbar), 0\right)\right\}$
\begin{equation}
\left\langle y'-y,
x-y\right\rangle
\leq\varepsilon\,||y'-y||\Vert x-y\Vert,\quad\paren{\forall y'\in  \Ball_\delta(\xbar)}
\paren{\forall y\in P_{\Omega}(y')}. \label{e:sup-reg}
\end{equation}
\end{defn}
\noindent Rewriting the above leads the the following equivalent 
characterization of super-regularity, which might be  more 
useful for our purposes.
\begin{propn}\cite[Proposition~4.4]{LewLukMal2009}\label{t:super-reg}The set $\Omega\subseteq\Euclid$
is \emph{super-regular} at $\xbar\in \Omega$ if and only if it is locally closed at
$\xbar$ and for every  $\varepsilon>0$ there exists $\delta>0$
such that
\begin{eqnarray}
&&\left\langle v,
x-y\right\rangle
\leq\varepsilon\,||v||\Vert x-y\Vert, \nonumber\\
&&\qquad\qquad\qquad\qquad\paren{\forall (x,v)\in 
\gph\ncone{\Omega}\cap\left(\Ball_{\delta}(\xbar)\times \Euclid\right)}\paren{ 
\forall y\in\Omega\cap\Ball_{\delta}(\xbar)}. \label{e:sup-reg2}
\end{eqnarray}
\end{propn}
To extend super-regularity to super-regularity at a distance, we employ the more general framework of 
{\em $\epsilon$-subregular sets} first introduced in \cite{Kruger2018}.  The following terminology 
follows \cite{DanLukTam19}.   
\begin{defn}[$\epsilon$-subregularity]\cite[Definition 2.2]{DanLukTam19}
	A set $\Omega \subset \Euclid$ is {\em $\epsilon$-subregular relative to $\Lambda\subset \Euclid$ at $\xbar$ for 
	$(x,v)\in \gph \ncone {\Omega}$} if it is locally closed at $\xbar$ and there
	exists an $\epsilon>0$ together with a neighborhood $U_\epsilon$ of $\xbar$ such that
	\begin{align}\label{eq: epsilon subregularity}
		\ip{v-(y'-y)}{y-x}\leq \epsilon\norm{v-(y'-y)}\norm{y-x}\quad (\forall y'\in \Lambda\cap U_\epsilon)(\forall y\in P_\Omega(y')).
	\end{align}
	$\Omega$ is {\em subregular relative to $\Lambda$ at $\xbar$ for $(x,v)\in \gph \ncone {\Omega}$}
	if it is locally closed and for all $\epsilon>0$ there exists $U_\epsilon$ such that 
	\eqref{eq: epsilon subregularity} holds.
\end{defn}
\begin{defn}[super-regularity at a distance]\label{defn: super-reg+}
	A set $\Omega\subset \Euclid$ is called {\em $\epsilon$-super-regular at a distance relative
	to $\Lambda\subset \Euclid$ at $\xbar$} if it is $\epsilon$-subregular relative
	to $\Lambda$ at $\xbar$ for all $(x,v)\in V_\epsilon$ where
	\begin{align}\label{eq:def V epsilon}
		V_\epsilon\equiv \set{(x,v)\in \gph \pncone {\Omega}}{x+v\in U_\epsilon,~x\in P_\Omega(x+v)},
	\end{align}
	and $U_\epsilon$ is a neighborhood of $\xbar$ with respect to an $\epsilon>0$.
	The set $\Omega$ is called {\em super-regular at a distance relative
	to $\Lambda$ at $\xbar$} if it is $\epsilon$-super-regular at a distance relative
	to $\Lambda$ at $\xbar$ for all $\epsilon>0$.
\end{defn}
Note that implicitly $U_\epsilon\cap \Lambda\neq\emptyset$ for all 
$\epsilon>0$.  
\begin{rem}[super-regularity at a distance relative to $\Euclid$ implies super-regularity]
	Being super-regular at a distance relative to $\Lambda=\Euclid$ at 
	some point $\xbar\in \Omega$ implies that the set is 
	super-regular at $\xbar$. To see this, let $\Omega$ 
	be super-regular at a distance relative to
	$\Lambda=\Euclid$ at $\xbar \in \Omega$. For fixed $
	\epsilon>0$ note that $(x,0)\in V_\epsilon$ for all
	$x\in \Omega\cap U_\epsilon$. With these, \eqref{eq: epsilon subregularity} becomes
	\begin{align}\label{eq: rem super-reg}
		\ip{y-y'}{y-x}\leq \epsilon\norm{y-y'}\norm{y-x}
	\end{align}
	for all $y'\in \Lambda\cap U_\epsilon$, $y\in P_\Omega(y')$ 
	and for all $x\in U_\epsilon\cap \Omega$.
	For sure, there exists an $\delta>0$ such that $\Ball_\delta\subset U_\epsilon$. 
	Moreover, since $\Lambda=\Euclid$ \eqref{eq: rem super-reg} holds for all
	$y'\in U_\epsilon$, $y\in P_\Omega(y')$ and for all $x\in U_\epsilon\cap \Omega$,
	which is by \cref{Def-super-reg} super-regularity of $\Omega$ at $\xbar$.
\end{rem} 
\begin{propn}[convex sets are super-regular at a distance]\label{prop: cxv set is super-reg at dist}
	Let $\Omega\subset \Euclid$ be convex and closed. Then $\Omega$ is super-regular at a distance
	relative to $\Lambda=\Euclid$ at any $\xbar\in \Omega$.
\end{propn}
\begin{proof}
	Fix $\xbar\in \Omega$. For convex sets $\Omega$ one has
	\begin{align*}
		\ip{v}{x-y}\leq 0\quad\paren{\forall x,y\in \Omega}\paren{\forall v\in \ncone{\Omega}(y)}.
	\end{align*}
	Thus, for any neighborhood $U_\epsilon\subset \Euclid$ of $\xbar$, $y'\in U$, $y\in P_\Omega(y')$, which implies
	that $y'-y\in \ncone \Omega (y)$, we deduce that $\ip{y'-y}{x-y}\leq 0$ and thus	
	\begin{align*}
		\ip{v-(y'-y)}{y-x}\leq 0\quad (\forall y'\in \Lambda \cap U_\epsilon)(\forall y\in P_\Omega(y')).
	\end{align*}
	This shows super-regularity at a distance of $\Omega$ relative to $\Euclid$ at all $\xbar\in \Omega$ as claimed.
\end{proof}
\begin{eg}[circle]\label{eg: circle super-reg}
	Consider the set 
	\begin{align*}
		\Omega\equiv \set{(x_1,x_2)\in \Rbb^2}{x_1^2+x_2^2=1}.
	\end{align*}
	This set is $\epsilon$-subregular relative to $\Lambda=P_\Omega^{-1}(\xbar)$ at any $\xbar\in \Omega$ for all
	$(\xbar, v )\in \gph \ncone \Omega$ with $\epsilon=0$ (which implies that $\Omega$ is in fact subregular relative to $\Lambda$
	for all $(\xbar, v )\in \gph \ncone \Omega$). 
    Indeed, for any  $\delta\in(0,1)$ we have, for any $y'\in \Lambda \cap \Ball_\delta(\xbar)$, that 
    $y\in P_\Omega(y')$ is given by $y=\xbar$ and  \eqref{eq: epsilon subregularity} specializes to 
    \begin{align*}
		\ip{v-\paren{y'-y}}{y-\xbar}=\ip{v-\paren{y'-y}}{\xbar-\xbar}=0 \quad (\forall y'\in \Lambda \cap \Ball_\delta(\xbar)  
                (\forall v\in \ncone \Omega (\xbar)).
	\end{align*}
	Moreover, the set $\Omega$ is super-regular at a distance relative to $\Lambda=P_\Omega^{-1}(\xbar)$ at 
	any $\xbar \in \Omega$. To see this, we will first show that $\Omega$ is $\epsilon$-super-regular
	at a distance relative to $P_\Omega^{-1}(\xbar)$ at $\xbar$ for any $\epsilon\in (0,0.5)$.
	Fix a $\epsilon\in(0,0.5)$ and set $\delta=2\epsilon$. For any $w\in \ncone \Omega(\xbar)$
	and $x\in \Omega\cap\Ball_\delta(\xbar)$ it holds 
	$\cos \angle \paren{w,x-\xbar}\leq \cos \angle \paren{-\xbar,x-\xbar}$. By the law of cosine we conclude
	$\cos \angle \paren{-\xbar,\xbar-x}=\norm{\xbar-x}/2<\delta/2=\epsilon$. Since 
	$v-\paren{y'-\xbar}\in \ncone \Omega(\xbar)$ for all $y'\in \Lambda \cap \Ball_\delta(\xbar)$,
    by the definition of the inner product on $\Rbb^2$ we deduce
	\begin{align*}
		&\ip{v-\paren{y'-\xbar}}{\xbar-x}\\
		=&\cos \angle \paren{v-\paren{y'-\xbar},\xbar-x}\norm{v-\paren{y'-\xbar}}\norm{\xbar-x}\\
		\leq &\epsilon \norm{v-\paren{y'-\xbar}}\norm{\xbar-x}
	     \qquad(\forall y'\in \Lambda \cap \Ball_\delta(\xbar))(\forall (x,v)\in V_\delta)
	\end{align*} 
        where 
	\begin{align*}
		V_\delta\equiv \set{(x,v)\in \gph \pncone {\Omega}}{x+v\in \Ball_\delta(\xbar),~x\in P_\Omega(x+v)},
	\end{align*}
   	which shows that $\Omega$ is $\epsilon$-super-regular
	at a distance relative to $P_\Omega^{-1}(\xbar)$ at $\xbar$ for any $\epsilon\in (0,0.5)$.
	Likewise, the same is true for any $\epsilon>0.5$ when taking a ball with radius 
	$\delta$ around $\xbar$, where $\delta <1$. Thus, $\Omega$ is super-regular
	relative to $P_\Omega^{-1}(\xbar)$ at $\xbar$.
	
	In fact, we can even enlarge our neighborhood from a ball to a tube in radial direction.
	Fix $\xbar\in \Omega$, $\epsilon>0$ and some $\delta\in (0,1)$ such that the above construction
	is satisfied. Then
	\begin{align*}
		U\equiv \bigcup_{\substack{z\in P_\Omega^{-1}(\xbar)\\ \norm{z}\geq 1}}\Ball_\delta(z)
	\end{align*}
	is a neighborhood for $\xbar$ such that $\epsilon$-super-regularity relative to
	$\Lambda=P_\Omega^{-1}(\xbar)$ is satisfied for $\Omega$. Fortunately, our violation $\epsilon$ will not 
	be worse compared to the neighborhood being a ball with radius $\delta$ around $\xbar$.
	This allows us to include more points in $\Lambda\cap U$ without violating 
	\eqref{eq: epsilon subregularity}.
	\begin{figure}
		\center
		\begin{tikzpicture}[scale=1.1]
			\draw[thick,->] (-4,0) -- (2.5,0);
			\draw[thick,->] (0,-2.5) -- (0,2.5);
			\draw (0,0) circle (2cm);
			\draw (1.4,1) node {$\Omega$};
			\draw[dashed] (-2,0) circle (1.5cm);
			\draw[dashed,-] (-4,1.5) -- (-2,1.5);
			\draw[dashed,-] (-4,-1.5) -- (-2,-1.5);
			\fill (-2,0) circle (0.05) node[above right] {$\bar{x}$};
			\fill (-3.2,0) circle (0.05) node[above right] {$y'$};
			\fill (-1.8793852415718,0.68404028665134) circle (0.05) node[below right] {$x$};
			\fill (-2.8190778623577,1.026060429977) circle (0.05) node[above right] {$x+v$};
			\draw [->] (-1.8793852415718,0.68404028665134) -- (-2.8190778623577,1.026060429977);
			\draw (-2.6,-0.5) node[right] {$\delta$};
			\draw (-2,0)-- (-3.1490666646785,-0.9641814145298);
		\end{tikzpicture}
		\caption{\cref{eg: circle super-reg}}
	\end{figure}
\end{eg}
\begin{propn}[characterization of super-regularity at a distance]\label{prop: charac super reg sets relative}$~$
	\begin{enumerate}[(i)]
		\item\label{prop: charac super reg sets relative part i}
 			A nonempty set $\Omega\subset \Euclid$ is $\epsilon$-super-regular 
 			at a distance relative to $\Lambda\subset \Euclid$ at $\xbar$ if and only if 
			there is a neighborhood $U_\epsilon$ of $\xbar$ such that 
			\begin{equation}
				\norm{x-y}^2\leq \epsilon\norm{\paren{y'-y}-\paren{x'-x}}\norm{x-y}+\ip{x'-y'}{x-y}
				 \quad (\forall y'\in U_\epsilon\cap \Lambda)(\forall y\in P_\Omega(y'))
			\end{equation}
			holds with $x'=x+v\in U_\epsilon$ for all $(x,v)\in V_\epsilon$ for $V_\epsilon$ defined by \eqref{eq:def V epsilon}.
		\item\label{prop: charac super reg sets relative part ii} 
			Let $\Omega \subset \Euclid$ be $\epsilon$-super-regular 
			at a distance relative to $\Lambda$ at $\xbar$. Then
			\begin{equation}
				\norm{x-y}\leq \epsilon \norm{\paren{y'-y}-\paren{x'-x}}+\norm{x'-y'}
				 \quad (\forall y'\in U_\epsilon\cap \Lambda)(\forall y\in P_\Omega(y'))
			\end{equation}
			holds with $x'=x+v\in U_\epsilon$ for all $(x,v)\in V_\epsilon$.
	\end{enumerate}
\end{propn}
\begin{proof}
	\eqref{prop: charac super reg sets relative part i}. 
	Let $\Omega\subset \Euclid$ be $\epsilon$-super-regular 
 	at a distance relative to $\Lambda\subset \Euclid$ at $\xbar$. Then, for fixed $\epsilon>0$, there exists a neighborhood $U_\epsilon$
 	of $\xbar$ such that for all $(x,v)\in V_\epsilon$ for $V_\epsilon$ defined by \eqref{eq:def V epsilon} and $x'=x+v\in U_\epsilon$ 
 	the following hold:
	\begin{align*}
		\norm{x-y}^2=\ip{x-y}{x-y}&=\ip{y'-y-\paren{x'-x}}{x-y}+\ip{x'-y'}{x-y}\\
		&\leq \epsilon\norm{y'-y-\paren{x'-x}}\norm{x-y}+\ip{x'-y'}{x-y}.
	\end{align*}
	This proves the first part of the equivalence. For the remaining one let $U_\epsilon$ be a neighborhood of $\xbar$ such that 
	\begin{equation}\label{eq:second part}
		\norm{x-y}^2\leq \epsilon\norm{\paren{y'-y}-\paren{x'-x}}\norm{x-y}+\ip{x'-y'}{x-y}
			\quad (\forall y'\in U_\epsilon\cap \Lambda)(\forall y\in P_\Omega(y'))
	\end{equation}
	holds with $x'=x+v\in U_\epsilon$ for all $(x,v)\in V_\epsilon$ for $V_\epsilon$ defined by \eqref{eq:def V epsilon}.
	\eqref{eq:second part} is equivalent to 
	\begin{equation*}
		\ip{y'-y-\paren{x'-x}}{x-y}+\ip{x'-y'}{x-y}\leq \epsilon\norm{\paren{y'-y}-\paren{x'-x}}\norm{x-y}+\ip{x'-y'}{x-y},
	\end{equation*}
	$(\forall y'\in U_\epsilon\cap \Lambda)(\forall y\in P_\Omega(y'))$ by the calculations made before. Subtracting $\ip{x'-y'}{x-y}$
	from both sides and inserting $v=x'-x$ yields
	\begin{equation*}
		\ip{y'-y-v}{x-y}\leq \epsilon\norm{\paren{y'-y}-v}\norm{x-y},
		\quad (\forall y'\in U_\epsilon\cap \Lambda)(\forall y\in P_\Omega(y')).
	\end{equation*}
	Reordering the left-hand side we deduce the definition of $\epsilon$-super-regular for $\Omega$ at $\xbar$
	\begin{equation*}
		\ip{v-(y'-y)}{y-x}\leq \epsilon\norm{\paren{y'-y}-v}\norm{x-y},
		\quad (\forall y'\in U_\epsilon\cap \Lambda)(\forall y\in P_\Omega(y')).
	\end{equation*}
	
	\eqref{prop: charac super reg sets relative part ii}. The second part follows from
	\eqref{prop: charac super reg sets relative part i} by applying the Cauchy-Schwarz 
	inequality to the vectors $x'-y'$ and $x-y$.
\end{proof} 
	

\section{Properties of \texorpdfstring{$\raar$}{T {DRlambda}} and Characterization of its Fixed Points}\label{s: charac of fixed points}
Our convergence analysis is based on the framework established in \cite{LukNguTam17} and relies on two essential 
properties of fixed point mappings.  The first property describes the {\em expansiveness} of the mapping, 
or the violation of nonexpansiveness.  This is called {\em almost averaging} in \cref{d:eh}.  This property
also implies {\em single-valuedness} of the fixed point mapping at its fixed points 
\cref{thm: singlevalued at fixed points}.  The characterization of the fixed points is established in \cref{thm:fixed points}.  
In Section \ref{s: conv analysis} we discuss the second property,  {\em metric subregularity} \cref{def:2.17}, 
which describes the (one-sided) Lipschitz continuity of the inverse of the fixed point 
mapping at its fixed points.   When specialized to set feasibility, this takes on the more geometric 
property of {\em subtransversality} of the collection of sets \cref{def:3.6}.   
\subsection{Almost Averaged Mappings}\label{s:aa mappings}
\begin{defn}[almost nonexpansive/averaged mappings, \cite{LukNguTam17}, Definition 2.2]\label{d:eh}
Let $D$ be a nonempty subset of $\Euclid$ and let $T$ be a (set-valued) mapping from $D$ to $\Euclid$.
\begin{enumerate}[(i)]
	\item $T$ is said to be \emph{pointwise almost nonexpansive} on $D$ at $y \in D$ if there exists a constant
		$\epsilon \in [0,1)$ such that
		\begin{align}\label{eq:18}
			\norm{x^+-y^+}\leq \sqrt{1+\epsilon}\norm{x-y}\quad \forall \ y^+\in Ty\text{ and }\forall x^+\in Tx
			\text{ whenever }x \in D.
		\end{align}
		If (\ref{eq:18}) holds with $\epsilon=0$ then $T$ is called \emph{pointwise nonexpansive} at $y$ on $D$.
			
		If $T$ is pointwise (almost) nonexpansive on $D$ at every point $y \in D$ with the violation constant 
		$\epsilon$, then $T$ is said to be \emph{(almost) nonexpansive on $D$ 
		(with violation $\epsilon$)}.
	\item $T$ is called \emph{pointwise almost averaged on $D$ at $y$} if there is an averaging constant 
		$\alpha\in (0,1)$
		and a violation constant $\epsilon\in [0,1)$ such 
		that the mapping $\tilde{T}$ defined by
		\begin{align*}
			\tilde{T}=\tfrac{1}{\alpha}T+\tfrac{(\alpha-1)}{\alpha}\Id
		\end{align*}
		is almost nonexpansive at $y$ with violation $\epsilon/\alpha$.
	\end{enumerate}
\end{defn}
\begin{propn}[characterization of almost averaged operators]\label{prop: characterization of aa mappings}
	Let $\mmap T \Euclid \Euclid$, $U \subset \Euclid$ and let $\alpha \in (0,1)$. The
	following are equivalent.
	\begin{enumerate}[(i)]
		\item $T$ is pointwise almost averaged at $y$ on $U$ with violation $\epsilon$ and averaging constant $\alpha$.
		\item\label{eq:aa ii} $\paren{1-\frac{1}{\alpha}}\id+\frac{1}{\alpha}T$ is pointwise almost nonexpansive at $y$ on $U$ with
		violation $\epsilon/\alpha$.
		\item For all $x\in U$, $x^+\in T(x)$ and $y^+\in T(y)$ it holds that
		\begin{equation}
			\norm{x^+-y^+}^2\leq (1+\epsilon)\norm{x-y}^2-\frac{1-\alpha}{\alpha}\norm{\paren{x-x^+}-\paren{y-y^+}}.
		\end{equation}
	\end{enumerate}
	Therefore, if $T$ is pointwise almost averaged at $y$ on $U$ with violation $\epsilon$ and averaging constant
	$\alpha$ then $T$ is pointwise almost nonexpansive at $y$ on $U$ with violation at most $\epsilon$.
\end{propn}
\begin{proof}
	The proof of this statement can be found in \cite[Proposition 2.1]{LukNguTam17}.
\end{proof}
In terms of the above \cref{d:eh}, {\em pointwise firmly nonexpansive} mappings 
are pointwise averaged mappings with averaging constant $\alpha=1/2$.  
\begin{propn}[compositions of averages of averaged operators]\label{prop:2.10}
	Let $ \mmap {T_j} \Euclid \Euclid $ for $j=1,2, \dots, m$ be pointwise almost averaged on $U_j$ at all 
	$y_j \in S_j \subset \Euclid$ with violation $\epsilon_j$ and averaging constant $\alpha_j\in (0,1)$ where
	$U_j \supset S_j$ for $j=1,2,\dots, m$.
	\begin{enumerate}[(i)]
		\item \label{prop:2.10 i}If $U \equiv U_1=U_2=\cdots =U_m$ and $S\equiv S_1 =S_2=\cdots =S_m$ then the weighted mapping 
			$T\equiv \sum_{j=1}^m w_jT_j$ with weights $w_j\in [0,1], \ \sum_{j=1}^mw_j=1$, is pointwise almost
			averaged at all $y \in S$ with violation $\epsilon=\sum_{j=1}^m w_j\epsilon_j$ and averaging
			constant $\alpha=\max_{j=1,2, \dots,m}\klam{\alpha_j}$ on $U$.
		\item \label{prop:2.10 ii}If $T_jU_j\subseteq U_{j-1}$ and $T_jS_j\subseteq S_{j-1}$ for $j=2,3, \dots,m$, then the composite
			 mapping $T\equiv T_1\circ T_2\circ\cdots\circ T_m$ is pointwise almost nonexpansive at all 
			 $y \in S_m$ on $U_m$ with violation at most
			 \begin{align}\label{eq:29}
				 	\epsilon=\prod_{j=1}^m\paren{1+\epsilon_j}-1.
			 \end{align}
		\item \label{prop:2.10 iii}If $T_jU_j\subseteq U_{j-1}$ and $T_jS_j\subseteq S_{j-1}$ for $j=2,3, \dots,m$, then the composite
			 mapping $T\equiv T_1\circ T_2\circ\cdots\circ T_m$ is pointwise almost averaged at all $y \in S_m$
			 on $U_m$ with violation at most $\epsilon$ given by (\ref{eq:29}) and averaging constant at least 
			 \begin{align}
			 	\alpha=\frac{m}{m-1+\frac{1}{\max_{j=1,2,\dots,m}\klam{\alpha_j}}}.
			 \end{align}
	\end{enumerate}
\end{propn}
\begin{proof}
	The proof of this statement can be found in \cite[Proposition 2.4]{LukNguTam17}.
\end{proof}
\cref{defn: super-reg+} allows us to get pointwise almost nonexpansivity
of the projector and reflector on a neighborhood of a point in $\Omega$ relative to points not in $\Omega$.
This is of particular interest for us, since the fixed points of $\raar$ will be (depending on $\lambda<1$)
in neither of the sets $A$ and $B$ if the problem is inconsistent (see \cref{thm:fixed points}, where we do not require that $A\cap B\neq \emptyset$).
\begin{propn}[regularity of projectors and reflectors at a distance]\label{prop: projector and reflector on sup reg set}
	Let $\Omega\subset \Euclid$ be nonempty and closed, and let $U$ be a neighborhood of $\xbar \in \Omega$.
	Let $\Lambda\equiv P_\Omega^{-1}(\xbar)\cap U$.
	If $\Omega$ is $\epsilon$-super-regular at a distance at $\xbar$ relative to $\Lambda$ 
	with constant $\epsilon$ on the neighborhood $U$, then the following hold.
	\begin{enumerate}[(i)]
		\item \label{P is ane rel}If $\epsilon\in [0,1)$, then the projector $P_\Omega$ is pointwise almost nonexpansive 
			at each $y'\in \Lambda$ with violation $\epsilontilde$ on $U$ for 
			$\epsilontilde\equiv 4\epsilon/\paren{1-\epsilon}^2$. That is, at each 
			$y'\in \Lambda$
			\begin{align*}
				\norm{x-y}\leq \sqrt{1+\epsilontilde}\norm{x'-y'}
				=\frac{1+\epsilon}{1-\epsilon}\norm{x'-y'} \quad 
				\paren{\forall x'\in U}\paren{\forall x\in P_\Omega(x')}\paren{\forall y\in P_\Omega(y')}.
			\end{align*}
		\item \label{P is afne rel}If $\epsilon\in [0,1)$, then the projector $P_\Omega$ is pointwise almost firmly 
			nonexpansive at each $y'\in \Lambda$ with violation $\epsilontilde_2$ on 
			$U$ for 
			$\epsilontilde_2\equiv 4\epsilon(1+\epsilon)/\paren{1-\epsilon}^2$. That is,
			at each $y'\in \Lambda$
			\begin{align*}
				\norm{x-y}^2+\norm{\paren{x'-x}-\paren{y'-y}}^2\leq \paren{1+\epsilontilde_2}\norm{x'-y'} 
				\quad\paren{\forall x'\in U}\paren{\forall x\in P_\Omega(x')}\paren{\forall y\in P_\Omega(y')}. 
			\end{align*}
		\item \label{R is ane rel}The reflector $R_\Omega$ is pointwise almost nonexpansive at each $y'\in \Lambda$
			with violation $\epsilontilde_3\equiv 8\epsilon(1+\epsilon)/(1-\epsilon)^2$ 	
			on $U$. That is, for all $y'\in \Lambda$
			\begin{align*}
				\norm{x-y}\leq \sqrt{1+\epsilontilde_3}\norm{x'-y'} \quad 
				\paren{\forall x'\in U}\paren{\forall x\in R_\Omega(x')}\paren{\forall y\in R_\Omega(y')}.
			\end{align*}
	\end{enumerate}
\end{propn}
\begin{proof}
    Our proof follows that of \cite[Theorem 3.1]{LukNguTam17}.  
	Before proving each of the statements individually, note the following.
	Take any $x'\in U$. Then for each $x\in P_\Omega(x')$ we have
	$\paren{x,x'-x}\in \gph \pncone \Omega \subset \gph \ncone \Omega$. Moreover,
	 by construction $\paren{x, x'-x}\in V_\epsilon$ where $V_\epsilon$ is defined by \eqref{eq:def V epsilon}.

	\eqref{P is ane rel}. Choosing $x'\in U$ and $x\in P_\Omega(x')$ we get 
	$\paren{x,x'-x}\in\gph \pncone \Omega\subset \gph \ncone \Omega$. Applying 
	\cref{prop: charac super reg sets relative}\eqref{prop: charac super reg sets relative part ii}
	yields
	\begin{align*}
		\norm{y-x}\leq \epsilon\norm{\paren{x'-x}-\paren{y'-y}}+\norm{y'-x'}
	\end{align*}
	whenever $y'\in U\cap\Lambda$ and $y\in P_\Omega(y')$. Exploiting the triangle inequality
	we deduce
	\begin{align*}
		\norm{y-x}\leq \epsilon\paren{\norm{y'-x'}+\norm{y-x}}+\norm{y'-x'}
	\end{align*}
	and thus conclude the claimed result.
	
	\eqref{P is afne rel}. By super-regularity at a distance relative to $\Lambda$ of $\Omega$ and 
	\cref{prop: charac super reg sets relative}\eqref{prop: charac super reg sets relative part i}
	we have
	\begin{align*}
		&\norm{x-y}^2+\norm{\paren{x'-x}-\paren{y'-y}}^2\\
		=&2\norm{x-y}^2+\norm{x'-y'}^2-2\ip{x'-y'}{x-y}\\
		\leq &\norm{x'-y'}^2+2\epsilon\norm{\paren{x'-x}-\paren{y'-y}}\norm{x-y},
	\end{align*}
	for $\paren{x, x'-x}\in V_\epsilon$ and $y'\in U\cap \Lambda$, $y \in P_\Omega(y')$. Together
	with the triangle inequality this implies
	\begin{align*}
		&\norm{x-y}^2+\norm{\paren{x'-x}-\paren{y'-y}}^2\\
		\leq &\norm{x'-y'}^2+2\epsilon\paren{\norm{x'-y'}+\norm{x-y}}\norm{x-y}.
	\end{align*}
	Using part \eqref{P is ane rel} we deduce
	\begin{align}\label{showing P fne}
		&\norm{x-y}^2+\norm{\paren{x'-x}-\paren{y'-y}}^2\nonumber\\
		\leq &\paren{1+4\epsilon\frac{1+\epsilon}{\paren{1-\epsilon}^2}}\norm{x'-y'}^2
	\end{align}
	for all $\paren{x, x'-x}\in V_\epsilon$ and for all $y \in P_\Omega(y')$ at each
	$y'\in U\cap \Lambda$. Since, as mentioned in the beginning, for all
	$x'\in U$ it holds that $\paren{x, x'-x}\in V_\epsilon$ for all $x\in P_\Omega(x')$,
	\eqref{showing P fne} holds at each $y'\in \Lambda=\Lambda\cap U$ for
	all $x\in P_\Omega(x')$ whenever $x'\in U$. By 
	\cref{prop: characterization of aa mappings} with $\alpha=1/2$ we conclude
	that $P_\Omega$ is pointwise almost firmly nonexpansive at each $y'\in \Lambda$ with violation
	$4\epsilon\paren{1+\epsilon}/\paren{1-\epsilon}^2$ on $U$.

	\eqref{R is ane rel}. By \eqref{P is afne rel} $P_\Omega$ is pointwise almost 
	firmly nonexpansive	at each $y'\in \Lambda$ with violation 
	\[4\epsilon\paren{1+\epsilon}/\paren{1-\epsilon}^2\] on $U$. Thus, by 
	\cref{prop: characterization of aa mappings} the reflector, $R_\Omega\equiv 2P_\Omega-\Id$,
	is pointwise almost nonexpansive at each $y'\in \Lambda$ with violation
	$8\epsilon\paren{1+\epsilon}/\paren{1-\epsilon}^2$ on $U$.
\end{proof}	

\subsection{\texorpdfstring{$\raar$}{T {DRlambda}} is Almost Averaged at \texorpdfstring{$\Fix \raar$}{Fix T {DRlambda}}}\label{s:raar near fix pts}
For general multivalued mappings $\mmap{T}{\Euclid}{\Euclid}$ the set of fixed points is defined as 
\begin{equation}\label{d:Fix T}
   	\Fix T\equiv \set{x\in\Euclid}{x\in T(x)}.
\end{equation}
Note that, by this definition, the set $T(x)$ need not consist entirely of fixed points (see \cite[Example 2.1]{LukNguTam17}).
If $T$ is pointwise almost averaged, however, the mapping 
$T$ is single-valued on its fixed point set.
\begin{propn}\cite[Proposition 2.2 ]{LukNguTam17}\label{t:single-valued paa}
    If $\mmap{T}{\Euclid}{\Euclid}$ is pointwise almost nonexpansive on $D\subseteq\Euclid$ at $\xbar\in D$ with 
    violation $\varepsilon\geq 0$, then $T$ is single-valued at $\xbar$. In particular, if $\xbar\in\Fix T$ 
	then $T(\xbar)=\{\xbar\}$.
\end{propn}     
The mapping $\raar$ is a composition and convex combination of projectors and reflectors. The almost averaging 
property is preserved under compositions and convex combinations of pointwise almost averaged
mappings, as we have seen in \cref{prop:2.10}.  
\begin{lem}\label{th:single-valued projector}
	Let $\xbar\in \Euclid$ and let $\Omega\subset\Euclid$ be super-regular at a distance relative to 
   	$\Lambda\subset P_\Omega^{-1}(\bar{\omega})$ at
	$\bar{\omega}$ where $\omegabar\in P_\Omega(\xbar)$ and $\xbar\in\Lambda$.   
    In addition, for each $\epsilon>0$, let $\xbar\in U_\epsilon(\omegabar)$ where $U_\epsilon(\omegabar)$ is a 
    neighborhood of $\omegabar$ on which \eqref{eq: epsilon subregularity} holds. Then $P_\Omega(\xbar)=\{\omegabar\}$. 
\end{lem}
\begin{proof}
   	For some fixed $\epsilon>0$, we get by the assumptions on super-regularity at a distance of $\Omega$ relative to $\Lambda$ and 
	\cref{prop: projector and reflector on sup reg set}\eqref{P is ane rel} 
	that there exists some neighborhood $U_{\epsilon}(\bar{\omega})$ such that
	$P_\Omega$ is pointwise almost nonexpansive at $\xbar\in \Lambda\cap U_\epsilon(\bar{\omega})$ on 
        $U_{\epsilon}(\bar{\omega})$ with violation
	$\epsilontilde=4\epsilon/(1-\epsilon)^2$. 
	This implies single-valuedness of $P_\Omega$ at $\xbar$ by \cref{t:single-valued paa},
	i.e. that $\{\bar{\omega}\}=P_\Omega(\xbar)$, as claimed.	
\end{proof}
\begin{thm}[$\raar$ is pointwise almost firmly nonexpansive at its fixed points]\label{thm: singlevalued at fixed points}
	Let $A$, $B$ be closed and nonempty, $\lambda\in (0,1)$ and $\xbar \in \Fix \raar\neq\emptyset$.
	Let $\bbar\in P_B(\xbar)$ and $\abar\in P_A(2\bbar-\xbar)$. Suppose that $B$ is super-regular at a distance 
    relative to $\Lambda_{\bar{b}}\equiv P_B^{-1}(\bar{b})$ at
	$\bar{b}$ and, likewise,  $A$ is super-regular at a distance relative to 
    $\Lambda_\abar\equiv P_A^{-1}(\bar{a}) $ at $\bar{a}$.
    Suppose, moreover, that  the following hold.
    \begin{enumerate}[(i)]
	\item\label{ass B} For each $\epsilon>0$, $\xbar\in U_\epsilon(\bbar)$ where $U_\epsilon(\bbar)$ is a 
        neighborhood of $\bbar$ on which \eqref{eq: epsilon subregularity} holds for $\epsilon$.
	\item\label{ass A} For each $\epsilon>0$, $2\bbar-\xbar\in U_\epsilon(\abar)$ where $U_\epsilon(\abar)$ is a 
        neighborhood of $\abar$ on which \eqref{eq: epsilon subregularity} holds for $\epsilon$. 
    	\item \label{ass 1} $R_B(\Lambda_{\bar{b}})\subset \Lambda_{\bar{a}}$.
    	\item \label{ass 2}$R_B(U_\epsilon(\bbar))\subset U_\epsilon(\abar)$ for all $\epsilon>0$.
    \end{enumerate} 
     Then, $\{\bar{b}\}=P_B(\xbar)$,  $\{\bar{a}\}=P_A(R_B(\xbar))$, $\raar$ is single-valued at $\xbar$, 
     and  for all $\epsilon>0$ there exists a neighborhood $U(B,\epsilon,\xbar)$ of $\bar{b}$ such that
     $\raar$ is pointwise almost firmly nonexpansive at $\xbar$ with violation at most $\epsilon$ on $U(B,\epsilon,\xbar)$.  
\end{thm}
Before we begin the proof of this statement, we would like to point out an important feature of our construction.  
The claimed pointwise almost averagedness of $\raar$ at $\xbar$ holds on open subsets containing 
both $\xbar$ and $\bbar=P_B(\xbar)$.  This follows from assumption \eqref{ass B}.  The conclusion of the 
theorem could have been equivalently stated:  {\em for all $\epsilon>0$ 
there exists a neighborhood $U(B, \epsilon,\xbar)$ of $\xbar$ 
such that $\raar$ is pointwise almost averaged at $\xbar$ with violation at most $\epsilon$ on $U$}.  
We have presented the statement with neighborhood $U(B, \epsilon,\xbar)$ containing $\bbar$ to emphasize the fact that 
the open sets on which the regularity of $\raar$ holds is constructed relative to points $\bbar$ {\em at a distance}
from the point of interest $\xbar\in \Fix \raar$.  The usual use of balls for neighborhoods is not the most convenient
or appropriate for this setting. 
\begin{proof}[Proof of \cref{thm: singlevalued at fixed points}]
    Under assumptions \eqref{ass B} and \eqref{ass A}, \cref{th:single-valued projector} yields  $\{\bar{b}\}=P_B(\xbar)$ and 
    $\{\bar{a}\}=P_A(R_B(\xbar))$, as claimed.  From this one can immediately conclude that $\raar$ is single-valued at $\xbar$.

	For any fixed $\epsilon_B>0$, we get by the assumptions on super-regularity at a distance of $B$ relative to $\Lambda_b$ and 
	\cref{prop: projector and reflector on sup reg set}\eqref{P is afne rel} 
	that there exists some neighborhood $U_{\epsilon_B}(\bbar)$ such that
	$P_B$ is pointwise almost firmly nonexpansive at $\xbar\in \Lambda_{\bar{b}}\cap U(B, \epsilon_B,\xbar)$ 
	on $U_{\epsilon_B}(\bbar)$ with violation $\epsilon_{P_B}=4\epsilon_B(1+\epsilon_B)/(1-\epsilon_B)^2$. 
	Note that this also shows that $P_B$ is pointwise almost nonexpansive at $\xbar$ on $U_{\epsilon_B}(\bbar)$.
	Similarly, 	by \cref{prop: projector and reflector on sup reg set}\eqref{R is ane rel},
    $R_B$ is pointwise almost nonexpansive at $\xbar$ with violation
	$\epsilon_{R_B}=8\epsilon_B(1+\epsilon_B)/(1-\epsilon_B)^2$ on $U_{\epsilon_B}(\bbar)$. Likewise, for any $\epsilon_A>0$ 
	there exists a neighborhood $U_{\epsilon_A}(\abar)$ of $\abar$ such that
    $R_A$ is pointwise almost nonexpansive at $\abar=2\bbar-\xbar$ with violation 
    $\epsilon_{R_A}=8\epsilon_A(1+\epsilon_A)/(1-\epsilon_A)^2$ on $U_{\epsilon_A}(\abar)$. 

	By \eqref{ass 1} and \eqref{ass 2}, the assumptions of \cref{prop:2.10}\eqref{prop:2.10 ii} are satisfied, hence
	we deduce that, for any fixed $\epsilon_A>0$ there exists a neighborhood $U(A, \epsilon_{R_AR_B},\xbar)$ such that $R_AR_B$ is pointwise 
	almost nonexpansive at $\xbar$ with violation at most
	$\epsilon_{R_AR_B}=\epsilon_{R_A}+\epsilon_{R_B}+\epsilon_{R_A}\epsilon_{R_B}$ on $U(B, \epsilon,\xbar)$.
	
	By \cref{prop: characterization of aa mappings}\eqref{eq:aa ii} we get that $1/2(R_AR_B+\Id)$ is almost firmly nonexpansive at $\xbar$ with violation $\epsilon_{R_AR_B}/2$ on
	$U_{\epsilon_B}(\bbar)$.
	Applying \cref{prop:2.10}\eqref{prop:2.10 i} yields pointwise almost firm nonexpansivity of
	$\raar$ at $\xbar$ on $U_{\epsilon_B}(\bbar)$ with violation at most 
	\begin{align*}
		\epsilon'=\lambda\epsilon_{R_AR_B}/2+(1-\lambda)\epsilon_{P_B}.
	\end{align*}
	
	Since the above properties hold for each $\epsilon_B>0$ and $\epsilon_A>0$, then given any $\epsilon>0$
    we can construct the neighborhoods above so that $\epsilon'\leq \epsilon$.  We  conclude that 
    for any $\epsilon>0$ there is a neighborhood $U(B,\epsilon,\xbar)$ such that $\raar$ is pointwise almost nonexpansive 
	at $\xbar$ on $U(B,\epsilon,\xbar)$ with violation at most $\epsilon$ (the corresponding 
	neighborhood $U(A, \epsilon_{R_AR_B},\xbar)$ of $\abar$ will be denoted by $U(A,\epsilon,\xbar)$). 
	This conclusion is consistent with the fact established above
    that $\raar$ is single-valued, which completes the proof. 
\end{proof}
\begin{cor}\label{th:JEFF the Brotherhood}
In the setting of \cref{thm: singlevalued at fixed points},  fix $\epsilonbar>0$ and  let
$U(B, \epsilonbar, \xbar)$ and $U(A, \epsilonbar, \xbar)$ be neighborhoods that
satisfy the assumptions \eqref{ass B}, \eqref{ass A} and \eqref{ass 2} such that $\raar$
is pointwise almost firmly nonexpansive at $\xbar$ with violation $\epsilonbar$ on $U(B, \epsilonbar, \xbar)$. 
Then, for all $\epsilon<\epsilonbar$ there 
exists a neighborhood $U(B, \epsilon, \xbar)$ and a neighborhood $U(A, \epsilon, \xbar)$ such that 
conditions \eqref{ass B}, \eqref{ass A} and \eqref{ass 2} hold 
in  addition to the inclusions $U(A, \epsilon, \xbar)\subset U(A, \epsilonbar, \xbar)$ and 
$U(B, \epsilon, \xbar)\subset U(B, \epsilonbar, \xbar)$. 
\end{cor}
\cref{th:JEFF the Brotherhood} implies that $\raar$ is pointwise almost firmly nonexpansive at 
$\xbar$ with violation $\epsilon$ on $U(B, \epsilon, \xbar)$. The strength of 
\cref{th:JEFF the Brotherhood}, however, is hidden in the proof given below and the explicit 
construction of the neighborhoods $U(B, \epsilon, \xbar)$ and $U(A, \epsilon, \xbar)$.
Thus, under the assumptions of \cref{thm: singlevalued at fixed points}, and given
the neighborhoods for some fixed violation $\epsilonbar$, we are always 
able to restrict these neighborhoods to smaller sets where \eqref{eq: epsilon subregularity}
holds with some violation smaller than $\epsilonbar$.
\begin{proof}[Proof of \cref{th:JEFF the Brotherhood}]
	Our approach to prove this statement is based
	on an explicit construction of the neighborhood $U(A, \epsilon, \xbar)$
	and $U(B, \epsilon, \xbar)$.
	
	Let $\epsilon<\epsilonbar$. \cref{thm: singlevalued at fixed points}\eqref{ass B} implies that there 
	exists a neighborhood $U(B, \epsilon, \xbar)$ of $\bbar$ where \eqref{eq: epsilon subregularity} holds
	such that $U(B, \epsilon, \xbar)\subset U(B, \epsilonbar, \xbar)$. 
	To see this, note that by \eqref{ass B} the existence of $U(B, \epsilon, \xbar)$
	is guaranteed and thus only $U(B, \epsilon, \xbar)\subset U(B, \epsilonbar, \xbar)$
	has to be proven. Let $\tilde{U}(B, \epsilon, \xbar)$ be a neighborhood for $\epsilon$
	where \eqref{eq: epsilon subregularity}  holds.
	Then \eqref{eq: epsilon subregularity}  is satisfied for $\epsilonbar$
	as well. Thus, $ U(B, \epsilon, \xbar)\equiv\tilde{U}(B, \epsilon,\xbar)\cap U(B, \epsilonbar, \xbar)$ is 
	a neighborhood of $\bbar$ where \cref{thm: singlevalued at fixed points}\eqref{ass 2} holds and \eqref{eq: epsilon subregularity} 
	is satisfied for both $\epsilon$ and $\epsilonbar$, which shows 
	that $U(B, \epsilon, \xbar)\subset U(B, \epsilonbar, \xbar)$ as required. 
	Next, applying the reflection on $B$ on both of neighborhoods $U(B, \epsilon, \xbar)$
	and $U(B, \epsilonbar, \xbar)$ yields
	\begin{align}\label{eq: mia}
   		R_B\paren{U(B, \epsilon, \xbar)}\subset R_B\paren{U(B, \epsilonbar, \xbar)}.
	\end{align}
	Let $\tilde{U}(A, \epsilon, \xbar)$ be a neighborhood of $\abar$ where
	\eqref{eq: epsilon subregularity} holds for $\epsilon$ such that
	\eqref{ass 2} of \cref{thm: singlevalued at fixed points} is satisfied. That is
	\begin{align}\label{eq: mia2}
   		R_B(U(B, \epsilon,\xbar))\subset \tilde{U}(A, \epsilon,\xbar).
   	\end{align}  
   	Combining \cref{thm: singlevalued at fixed points}\eqref{ass 2} for the neighborhoods $U(B, \epsilonbar, \xbar)$ and $U(A, \epsilonbar, \xbar)$
   	and \eqref{eq: mia} we deduce
   	\begin{align}
   	 	R_B\paren{U(B, \epsilon, \xbar)}\subset R_B\paren{U(B, \epsilonbar, \xbar)}\subset U(A,\epsilonbar, \xbar).
   	\end{align}
   	This and \eqref{eq: mia2} imply that
   	\begin{align}\label{eq: mia3}
   	 	R_B(U(B, \epsilon,\xbar))\subset \tilde{U}(A, \epsilon,\xbar)\cap U(A,\epsilonbar, \xbar).
   	\end{align}
   	Set 
   	\begin{align*}
   		U(A,\epsilon, \xbar)\equiv \tilde{U}(A, \epsilon,\xbar)\cap U(A,\epsilonbar, \xbar).
   	\end{align*}
   	Then, $U(A,\epsilon, \xbar)$ is neighborhood of $\abar $ where \eqref{eq: epsilon subregularity} 
   	holds with $\epsilon$, since it is a subset of $\tilde{U}(A, \epsilon,\xbar)$. Moreover, $U(B, \epsilon,\xbar)$
   	and $U(A,\epsilon, \xbar)$ satisfy \eqref{ass 2} of \cref{thm: singlevalued at fixed points} by \eqref{eq: mia3}. By the construction of
   	$U(A, \epsilonbar, \xbar)$ and the choice of $U(B, \epsilon,\xbar)$ both sets satisfy the inclusions
   	$U(A, \epsilon, \xbar)\subset U(A, \epsilonbar, \xbar)$ and 
	$U(B, \epsilon, \xbar)\subset U(B, \epsilonbar, \xbar)$. This completes the proof.	 
\end{proof}

\begin{eg}\label{eg:1}
	The following examples illustrate the assumptions of the above theorem.  For these examples
	it is easy to determine the sets of fixed points of the mapping $\raar$.  In Theorem \ref{thm:fixed points}
	we give a precise characterization of $\Fix\raar$.  More intuitively, the fixed points must lie 
	on lines containing {\em local best approximation points} between the sets.   
	\begin{enumerate}[(i)]
	\item\label{eg:1i} (convex sets with empty intersection). 
 		Let $A$ and $B$ be closed convex subsets of $\Euclid$. By \cref{prop: cxv set is super-reg at dist}
 		both sets are super-regular relative to $\Euclid$ at any of their points, i.e.
 		$\epsilon$-super-regular for all $\epsilon>0$. In fact, the violation can be set to $0$.  
		Thus, as long as $\Fix \raar\neq\emptyset$ the mappings $P_B$, $R_B$ and $R_A$ 
		are nonexpansive (i.e. no violation) at $\xbar$ on the whole space $\Euclid$ 
		by \cref{prop: projector and reflector on sup reg set} and we 
		can apply \cref{thm: singlevalued at fixed points}
		to conclude that $\raar$ is firmly nonexpansive at $\xbar$ on the neighborhood $\Euclid$.
        For instance, consider the two sets
 		\begin{align*}
 			A\equiv \set{x=(x_1,x_2)\in \Rbb^2}{x_1^2+x_2^2\leq 1}\text{ and } B\equiv \set{x=(x_1,x_2)\in \Rbb^2}{(x_1-3)^2+x_2^2\leq 1}.
 		\end{align*}
		The set of fixed points is given by the unique point
		\begin{align*}
			\Fix \raar =\klam{\xbar}=\klam{(2,0)-\frac{\lambda}{1-\lambda}(1,0)}
		\end{align*}		 		
 		for fixed $\lambda\in (0,1)$, and by the above discussion, we know that $\raar$ built 
        from the projections onto these sets is firmly nonexpansive.  
	\item\label{eg:1ii} (super-regular sets with empty intersection). Continuing with the 
        concrete example above, suppose that $A$ and $B$ are spheres instead of balls, 
		\begin{align*}
 			A\equiv \set{x=(x_1,x_2)\in \Rbb^2}{x_1^2+x_2^2= 1}\text{ and } B\equiv \set{x=(x_1,x_2)\in \Rbb^2}{(x_1-3)^2+x_2^2= 1}.
 		\end{align*}	
		The sets $A$ and $B$ are both non-convex, but still super-regular.
 		The set of fixed points is again given by the unique point
		\begin{align*}
			\Fix \raar =\klam{\xbar}=\klam{(2,0)-\frac{\lambda}{1-\lambda}(1,0)}
		\end{align*}		 		
 		for fixed $\lambda\in (0,1)$. As seen in \cref{eg: circle super-reg} both sets
 		are super regular relative to radial directions. 
		Thus, applying \cref{thm: singlevalued at fixed points} we deduce that $\raar$
		for some fixed $\lambda\in (0,1)$ is only almost firmly nonexpansive at $\xbar$ 
		on some neighborhood $U$. As noted before in \cref{eg: circle super-reg} the
		neighborhood should be rather chosen as a tube than the more conventional ball.
		Such a choice of neighborhoods is visualized in \cref{fig:thm raar single valued}.
		\begin{figure}[ht]
		\centering
		\begin{tikzpicture}]
			\draw (0,0) circle (2cm);
			\draw (0,1.5) node {$A$};
			\draw (6,0) circle (2cm);
			\draw (6,1.5) node {$B$};
			\fill (2,0) circle (0.05) node[above right] {$\bar{a}$}; 
			\fill (4,0) circle (0.05) node[below left] {$\bar{b}$}; 
			\fill (-3,0) circle (0.05) node[above right] {$\bar{x}$};
			\fill (11,0) circle (0.05) node[below] {$2\bar{b}-\bar{x}$}; 
			\draw[dashed] (2,1) arc(90:270:1);
			\draw[dashed] (2,1) -- (11.5,1);
			\draw[dashed] (2,-1) -- (11.5,-1);
			\draw[dashed] (4,-0.5) arc(-90:90:0.5);
			\draw[dashed] (4,0.5) -- (-3.5,0.5);
			\draw[dashed] (4,-0.5) -- (-3.5,-0.5);
			\draw (0,0) node {$U_\epsilon(\bar{b})$};
			\draw (9,0.5) node {$U_\epsilon(\bar{a})$};
		\end{tikzpicture}
		\caption{\cref{eg:1}\eqref{eg:1ii}}\label{fig:thm raar single valued}
		\end{figure}
 	\item\label{eg:1iii}(super-regular sets with nonempty intersection). Next we translate the
 		sets in \eqref{eg:1ii} so that they have exactly one point in common.
 		\begin{align*}
 			A\equiv \set{x=(x_1,x_2)\in \Rbb^2}{x_1^2+x_2^2= 1}\text{ and } B\equiv \set{x=(x_1,x_2)\in \Rbb^2}{(x_1-2)^2+x_2^2= 1}.
 		\end{align*}
 		The fixed point set then reduces to $\Fix \raar =\klam{\paren{1,0}}=A\cap B$. By \eqref{eg:1ii}
 		we know that the assumptions of \cref{thm: singlevalued at fixed points} are satisfied.
 		In contrast to \eqref{eg:1ii} the fixed point is in the intersection of both sets. Thus, 
 		balls as neighborhoods are enough to get pointwise almost firm nonexpansivity. We do not need tubes
 		to include points from a distance.
	\end{enumerate}
	The examples show that in case of closed balls and circles the assumptions are easily satisfied. Nonetheless,
	one has to take care of choosing neighborhoods in the right way to get a desired violation.
	\begin{figure}
	\center
	\begin{tikzpicture}[scale=0.8]
	\draw[fill=gray] (-10,0) circle (1cm);
	\draw (-9.4,0.3) node {$A$};
	\draw[fill=gray] (-7,0) circle (1cm);
	\draw (-6.4,0.3) node {$B$};
	\draw (-3,0) circle (1cm);
	\draw (-2.4,0.3) node {$A$};
	\draw (0,0) circle (1cm);
	\draw (0.6,0.3) node {$B$};
	\draw (4,0) circle (1cm);
	\draw (4.6,0.3) node {$A$};
	\draw (6,0) circle (1cm);
	\draw (6.6,0.3) node {$B$};
	\end{tikzpicture}
	\caption{\cref{eg:1}\eqref{eg:1i}-\eqref{eg:1iii}}
	\end{figure}
\end{eg}
\cref{eg:1}\eqref{eg:1i} yields the following specialization of \cref{thm: singlevalued at fixed points}.
\begin{cor}
	Let $\lambda \in (0,1)$ and $\Fix \raar \neq \emptyset$. If $A$ and $B$ are closed and convex, 
	$\raar$ is firmly nonexpansive	on $\Euclid$.	
\end{cor}
\begin{proof}
	Since $A$ and $B$ are both convex one has by \cref{prop: cxv set is super-reg at dist}
	that both sets are super-regular at a distance relative to $\Euclid$ at any of their points.
	Applying \cref{thm: singlevalued at fixed points} we deduce firm nonexpansivity
	of $\raar$ since the violation $\epsilon$ can be set to $0$, as seen in the proof 
	of \cref{prop: cxv set is super-reg at dist}.
\end{proof}

\subsection{Characterization of \texorpdfstring{$\Fix \raar$}{Fix T {DRlambda}}}
We collect some facts and identities that will be useful throughout.
\begin{lem}\label{lem:identities}
	Let $A$ and $B$ be closed and $\raar$ given by \eqref{raar} with $\lambda\in (0,1)$. Let 
	$x\in\Fix \raar\neq \emptyset$  such that $\raar$ is single-valued at $x$. Take $f\in P_B(x)$ and $y\equiv x-f$. 
	Then the following hold.
	\begin{enumerate}[(i)]
   		\item\label{lem:identities 1} $P_B(x)=\{f\}$, that is, $P_B$ is single-valued on $\Fix \raar$;
   		\item\label{lem:identities 2} $P_A\paren{R_B(x)}$ is single-valued, hence so is $R_A\paren{R_B(x)}$;
   		\item\label{eq:1} $P_A(2f-x)=P_A(R_B(x))$;
  		\item\label{eq:2}$T_{{DR}}(x)-x=P_A(R_B(x))-P_B(x)$;
   		\item\label{eq:3} $f+\frac{1-\lambda}{\lambda}y=P_A(2f-x)$;
   		\item\label{lem:identities e-Acvx} if $A$ is convex, then, for $e= P_A(f)$
			\begin{equation}\label{eq:e-Acvx}
   				P_A\paren{e+\tfrac{1}{1-\lambda}(f-e)}=e.
			\end{equation}
	\end{enumerate}
\end{lem}
\begin{proof}
	\eqref{lem:identities 1}-\eqref{lem:identities 2}. Since
	\begin{align*}
		\raar(x)&=\set{\frac{\lambda}{2}\paren{R_A(2b-x) +x}+\paren{1-\lambda}b}{ \ b \in P_B(x)},
	\end{align*}
	is just a single point, we conclude that $P_B(x)$ as well as $P_A\paren{R_B(x)}$ and $R_A\paren{R_B(x)}$ 
   	have to be single-valued, as claimed.
   	
	\eqref{eq:1}.  This is an easy implication of the single-valuedness of $P_B$ at $x$: 
	\begin{align*}
		P_A\paren{2f-x}=P_A\paren{2P_B(x)-x}=P_A\paren{R_B(x)}.
	\end{align*}
	
	\eqref{eq:2}.  This also follows from single-valuedness of $P_B$ at $x$:
	\begin{align*}
		T_{{DR}}(x)-x&=\frac{1}{2}\paren{R_A\paren{R_B(x)}+x}-x\\
		&=\frac{1}{2}\paren{R_A\paren{R_B(x)}}-\frac{1}{2}x\\
		&=P_A\paren{R_B(x)}-\frac{1}{2}R_B(x)-\frac{1}{2}x\\
		&=P_A\paren{R_B(x)}-P_B(x).	
	\end{align*}
	
	\eqref{eq:3}. To see this, note that 
	\begin{align*}
		&&&&x=\raar (x)&=\frac{\lambda}{2}\paren{R_A\paren{R_B(x)}+x}+\paren{1-\lambda}P_B(x)\\
		&&\iff&&\paren{1-\lambda}x&=\lambda\paren{T_{{DR}}(x)-x}+\paren{1-\lambda}P_B(x)\\
		&&\iff&&\paren{1-\lambda}\paren{x-P_B(x)}&=\lambda\paren{P_A\paren{R_B(x)}-P_B(x)}, 
	\end{align*}
	by part \eqref{eq:2}.  Hence, with $f=P_B(x)$, this yields 
	\begin{align*}
		\paren{1-\lambda}\paren{x-f}&=\lambda\paren{P_A\paren{2f-x}-f}
		\iff f+\frac{1-\lambda}{\lambda}y=P_A(2f-x),
	\end{align*}
	by the definition of $y$.  
	
	\eqref{lem:identities e-Acvx}.  This follows from the fact that $f-e\in \pncone{A}(e)$.  Since $A$ is convex, 
	then all points in $e+\pncone{A}(e)$ project back to $e$.  
\end{proof}
\begin{rem}
	Note that \eqref{lem:identities 1} and \eqref{lem:identities 2} of \cref{lem:identities} 
	together at some point $x\in \Euclid$ are equivalent to the single-valuedness of $\raar$ at $x$. 
\end{rem}
\begin{thm}[fixed points]\label{thm:fixed points}
	Let $A, B \subset \Euclid$ both be closed and let $\lambda\in (0,1)$. 
	Let $\raar$ be single-valued at its fixed points on an open set $U\subset \Euclid$. Then
	\begin{align}\label{eq:fixed points}
		\Fix \raar \cap U\subset \Mcal\equiv \set{f-\frac{\lambda}{1-\lambda}\paren{f-e}}{f \in P_B\paren{f-\frac{\lambda}{1-\lambda}\paren{f-e}},\text{ and } \ e \in P_A(f)}\cap U.
	\end{align}
	The inclusion is tight if $e \in P_A\paren{f+\frac{\lambda}{1-\lambda}(f-e)}$ is true for the right-hand side.
\end{thm}
\begin{proof}
	Let $x \in \Fix \raar\cap U$. By the assumptions $\raar$ is  
	single-valued at $x$, and hence the results in \cref{lem:identities} can be applied. As before denote by $f$ the projection $P_B(x)$.
	Reformulating \cref{lem:identities}\eqref{eq:3} yields the desired form of the fixed point $x$.
	\begin{align}
		x \in \Fix \raar &&\Longrightarrow&&f+\frac{1-\lambda}{\lambda}y=&P_A(2f-x)\notag\\
		&&\iff&&f+\frac{1-\lambda}{\lambda}\paren{x-f}=&P_A(2f-x)\notag\\
		&&\iff&&x=&\frac{\lambda}{1-\lambda}P_A(2f-x)-\frac{2\lambda-1}{1-\lambda}f\notag\\
		&&\iff&&x=&f-\frac{\lambda}{1-\lambda}\paren{f-P_A(2f-x)}.\label{eq:4}
	\end{align}
	Comparing with \eqref{eq:fixed points} we have to show that $P_A(f)=P_A\paren{2f-x}$. 
	
	By \cref{lem:identities}\eqref{lem:identities 1} and \eqref{lem:identities 2} we know that $P_B=\klam{f}$ is single-valued 
	as well as $P_A(2f-x)$. Reformulating $x \in \Fix \raar$ again yields
	\begin{align}
		x \in \Fix \raar&&\iff&&x&=\frac{\lambda}{2}\paren{R_A\paren{R_B(x)}+x}+\paren{1-\lambda}P_B(x)\notag\\
		&&\iff&& f &= \lambda P_A(2f-x) + (1-\lambda)(2f-x).
	\end{align}
	Thus, $f$ is a convex combination of $P_A(2f-x)$ and $2f-x$. The projection of $f$ onto $A$, therefore, has to be in $P_A(2f-x)$,
	which was a single point. That is, $P_A(f)\subset \klam{P_A(2f-x)}$, which immediately implies $P_A(f)=P_A(2f-x)$.

	Using this fact, then \eqref{eq:4} becomes
	\begin{align*}
		x= f-\frac{\lambda}{1-\lambda}\paren{f-P_A(f)}.
	\end{align*}
	Finally, \eqref{eq:fixed points} follows from the fact that $f=P_B(x)$,
	since $x$ is a fixed point.

	It remains to show that the inclusion is in fact an equality when 
	\begin{align*}
		e \in P_A\paren{f+\frac{\lambda}{1-\lambda}(f-e)}
	\end{align*} 
	for $e\in P_A(f)$.
	To see this, let $\xtilde\in  \Mcal\cap U$ in \eqref{eq:fixed points}.  Then  
	$\xtilde\equiv f-\frac{\lambda}{1-\lambda}\paren{f-e}$ 
	for some $e\in P_A(f)$ and $f\in P_B(\xtilde)$, and
	\begin{align*}
		\xtilde -\raar \xtilde  =& \xtilde -\frac{\lambda}{2}\paren{R_A\paren{R_B(\xtilde )}+\xtilde }-\paren{1-\lambda}P_B(\xtilde )\\
		=& \lambda \xtilde -\frac{\lambda}{2}\paren{2P_A\paren{R_B(\xtilde )}-2P_B(\xtilde )+2\xtilde }+\paren{1-\lambda}\paren{\xtilde -P_B(\xtilde )}\\
		=& -\lambda\paren{P_A\paren{R_B(\xtilde )}-P_B(\xtilde )}+\paren{1-\lambda}\paren{\xtilde -P_B(\xtilde )}\\
		\ni& -\lambda\paren{P_A\paren{R_B(\xtilde )}-f}+\paren{1-\lambda}\paren{\xtilde -f}\\
		=& -\lambda\paren{P_A\paren{R_B(\xtilde )}-f}-\lambda\paren{f-e}\\
		=& -\lambda\paren{P_A\paren{R_B(\xtilde )}}+\lambda e.
	\end{align*}
	Thus $0\in \xtilde -\raar \xtilde  $ if and only if $e \in P_A\paren{R_B(\xtilde )}$, which is equivalent 
	to $e\in P_A\paren{f+\frac{\lambda}{1-\lambda}(f-e)}$.
	This concludes the proof.
\end{proof}
\begin{rem}\label{r:Fix char}
	\begin{enumerate}[(i)]
		\item\label{r:Fix char i} Note that $f+\tfrac{\lambda}{1-\lambda}(f-e) = e + \tfrac{1}{1-\lambda}(f-e)$, so that for any 
           	$e\in P_A(f)$, $f-e$ is in the normal cone to $A$ at $e$.  It follows immediately that, if $A$ is convex 
            then $P_A\paren{e + \tfrac{1}{1-\lambda}(f-e)}=e$ for all $\lambda\in (0,1)$ and, by 
            \cref{thm:fixed points} the inclusion \eqref{eq:fixed points} is in fact equality for all $\lambda$ for which 
            $f\in P_B\paren{f-\tfrac{\lambda}{1-\lambda}\paren{f-e}}$.  
            Compare this to the statement in \cite[Lemma 3.8]{Luke2008} 
            where the tight fixed point characterization holds for $\lambda \in [0,1/2]$. This is due to the slightly different 
            characterization.  The statement in \cite{Luke2008} that $f$ is a 
            local best approximation point is actually incorrect. Where our description includes 
			$f \in P_B\paren{f-\frac{\lambda}{1-\lambda}\paren{f-e}},\text{ and } \ e \in P_A(f)$, the version 
			in \cite[Lemma 3.8]{Luke2008} states that $f$ is a local best approximation point \cite[Definition 3.3]{Luke2008}.  
            Instead, what is needed to correct the 
            statement is $f\in P_B(P_A(f))$, and such a point need not be a local best approximation point.    
            To see this, consider a unit circle in $\mathbb{R}^2$ centered at the origin
            and a horizontal line passing through the point $(0,3/4)$.  For the fixed point mapping $\raar$ 
            with $A$ the line and $B$ the circle, the point $\paren{0, 1-\tfrac{\lambda}{4(1-\lambda)}}$ is a fixed point for all 
            $\lambda\in(0,4/5)$.  But the corresponding points $f=(0,1)$ and $e=(0,3/4)$ are not local best approximation points.  
		\item The condition $e\in P_A\paren{f+\frac{\lambda}{1-\lambda}(f-e)}$ is easier to interpret with the identity 
            $f+\tfrac{\lambda}{1-\lambda}(f-e) = e + \tfrac{1}{1-\lambda}(f-e)$.  As $\lambda\nearrow 1$ this vector 
            receeds from $A$ in the direction normal to $A$ at $e$.   The larger the neighborhood
			on which the projection on $A$ exists and is single-valued, the larger $\lambda$ can be before 
            $e\notin P_A\paren{f+\frac{\lambda}{1-\lambda}(f-e)}$.  If $A$ is convex, then $\lambda$ can be arbitrarily close 
            to $1$.  Still, $\lambda$ may need to be bounded away from $1$ in order to ensure the other condition in the fixed point 
            characterization \eqref{eq:fixed points}, namely $f\in P_B\paren{f-\tfrac{\lambda}{1-\lambda}(f-e)}$.     
		\item\label{r:Fix char iii} By \cref{thm: singlevalued at fixed points} we know that $\raar$ is single-valued at its fixed points
			if both $A$ and $B$ are super-regular at a distance and assumptions \eqref{ass B}-\eqref{ass 2} of 
			\cref{thm: singlevalued at fixed points} hold.  
			The local gap $f-P_A(f)$ is therefore unique. In \cite{Luke2008} uniqueness of such gap vectors was an 
		    assumption of the convergence analysis. Our results show that we can remove this assumption.
	\end{enumerate}
\end{rem}
\begin{cor}[fixed points of DR$\lambda$ and the corresponding gap]\label{cor: fixed point and its gap}
	In the setting of \cref{thm:fixed points} let $x \in \Fix \raar \cap U$. Then 
	\begin{align*}
		\{x\} &= P_B(x)-\frac{\lambda}{1-\lambda}\paren{P_B(x)-P_A\paren{P_B(x)}}.
	\end{align*}
\end{cor}
\begin{proof}
	The result follows directly from the proof of \cref{thm:fixed points}.
\end{proof}
In our statements we require that $\Fix\raar \neq \emptyset$. Although this assumption is very strong, it is not 
very restrictive and is satisfied under the assumption of compactness of one of the underlying sets and convexity of both sets. 
\begin{propn}[convexity and compactness imply nonempty fixed point set]
	Let $\lambda \in (0,1)$. If $A$ and $B$ are convex and closed, and $A$ is bounded, then $\Fix \raar\neq \emptyset$.
	Moreover, $\Fix\raar=\Mcal$ where $\Mcal$ is given by \eqref{eq:fixed points}.
\end{propn}
\begin{proof}
    The proof follows the pattern of proof in \cite[Lemma 2.1]{Luke2008} which establishes existence of fixed points for 
    $\raar$ by first showing the existence of fixed points of the alternating projections mapping $T\equiv P_AP_B$.  
    To see this, note that $T$ is nonexpansive since the projectors $P_A$ and $P_B$ are nonexpansive, and 
	the composition of nonexpansive mappings is nonexpansive by a similar argument as made in
	\cref{eg:1}\eqref{eg:1i}. Note that $U=\Euclid$. Existence of fixed points of $T$ is then an easy consequence of 
	\cite[Theorem 2]{Browder66}, which requires that one of the sets, $A$ or $B$ be compact.  
	
	To show the tight characterization of $\Fix T$, let $e\in \Fix T$.  Then 
	$P_B(e)=f$ and $P_A(f)=e$ and $\raar$ is , by convexity, single-valued.  By \cref{r:Fix char}\eqref{r:Fix char i} we 
	have  $f+\tfrac{\lambda}{1-\lambda}(f-e) = e + \tfrac{1}{1-\lambda}(f-e)$ and 
	$P_A\paren{e + \tfrac{1}{1-\lambda}(f-e)}=e$ for all $\lambda\in (0,1)$.  Moreover, for all 
    such $\lambda$ we have $f= P_B\paren{f-\tfrac{\lambda}{1-\lambda}\paren{f-e}}$. Together, for $\xbar=f-\frac{\lambda}{1-\lambda}(f-e)$, this yields
    \begin{align*}
    	\raar(\xbar)
    	=&\frac{\lambda}{2}\paren{R_A(R_B(\xbar))+\xbar}+(1-\lambda)P_B(\xbar)\\
    	=&\frac{\lambda}{2}\paren{R_A(2f-\xbar)+\xbar}+(1-\lambda)f\\
    	=&\frac{\lambda}{2}\paren{2P_A(e+\frac{1}{1-\lambda}(f-e))-2f+\xbar+\xbar}+(1-\lambda)f\\
    	=&\frac{\lambda}{2}\paren{2e-2f+2\xbar}+(1-\lambda)f\\
    	=&\lambda\paren{f-\frac{\lambda}{1-\lambda} (f-e)}+(1-\lambda)f\\
    	=&f-\lambda\frac{\lambda}{1-\lambda} (f-e)=\xbar.
    \end{align*}
    Now, applying \cref{thm:fixed points} immediately yields $\Fix\raar=\Mcal$ where $\Mcal$ is given by \eqref{eq:fixed points}.  
	This completes the proof. 
\end{proof}
The above result on existence relies heavily on convexity.   The next example shows that it is quite easy to construct 
a scenario where $\raar$ has no fixed points.  
\begin{eg}[empty fixed point set]\label{ex:circle point}
	Let $A$ be the unit circle in $\Rbb^2$, i.e. 
	\begin{align*}
		A\equiv \set{x=(x_1,x_2) \in \Rbb^2}{x_1^2+x_2^2=1},
	\end{align*}
	and	$B$ its origin, i.e. $B\equiv \klam{\paren{0,0}}$. In this setting the fixed point set of $\raar$ is 
	empty for all $\lambda \in (0,1)$.
	To prove this we will show by a case distinction that the fixed point set of $\raar$ is empty.
	
	First, note that the projectors and reflectors involved in $\raar$
	are given by
	\begin{align*}
		P_B(x)=&\paren{0,0} \quad \forall x \in \Rbb^2\\
		P_A(x)=&\begin{cases}
			\frac{x}{\norm{x}}\quad &\forall x \in \Rbb^2\setminus\paren{0,0},\\  
			A \quad &\text{for } x=\paren{0,0}.
			\end{cases}
	\end{align*}
	Now, let $x=(0,0)$. Then
	\begin{align*}
		\raar(x)=\frac{\lambda}{2}\paren{R_AR_B(x)+x}+(1-\lambda)P_B(x)=\frac{\lambda}{2}\paren{R_A(x)}=\lambda A.
	\end{align*}
	Thus, $x=(0,0)$ cannot be a fixed point of $\raar$. For the other case let $x\neq(0,0)$. Then
	\begin{align*}
		\raar(x)=\frac{\lambda}{2}\paren{R_AR_B(x)+x}+(1-\lambda)P_B(x)=\frac{\lambda}{2}\paren{R_A(-x)+x}=\lambda \paren{x-\frac{x}{\norm{x}}}.
	\end{align*}
	If $x$ is a fixed point of $\raar$, that is $x=\raar(x)$, the following has to hold
	\begin{align*}
		x=\lambda \paren{x-\frac{x}{\norm{x}}},
	\end{align*}
	which is equivalent to
		\begin{align*}
		\frac{1-\lambda}{\lambda}x=-\frac{x}{\norm{x}}.
	\end{align*}
	But this is only satisfied when $x=(0,0)$, a contradiction. From which we conclude 
	that $x\notin \Fix \raar$, and therefore $\Fix \raar=\emptyset$.
	\begin{figure}[ht]
		\centering
		\begin{tikzpicture}
			\draw[thick] (0,0) circle (2cm);
			\draw (1.8,1.8) node {\Large $A$};
			\fill (0,0) circle (0.05) node[above right] {\Large $B$};
			\fill (0,1) circle (0.05) node[above right] {$x$};
			\fill (0,-1) circle (0.05) node[below left] {$R_B(x)$};
			\fill (0,-0.8) circle (0.05) node[above right] {$T_{\mathrm{DR}\lambda}(x)$};
			\fill (0,-2) circle (0.05) node[below] {$P_A\left(R_B(x)\right)$};
		\end{tikzpicture}
		\caption{\cref{ex:circle point} for a point $x\in \Rbb^2$ and $\lambda=0.8$.}\label{f:ex1}
	\end{figure}
\end{eg}
The following proposition provides a comparison of the fixed points for $\raar$ for different values of $\lambda$.
\begin{propn}[monotonicity of $\Fix \raar$ with respect to $\lambda$]
	Let $A$ and $B$ be both closed subsets of $\Euclid$, and $\lambda_1, \lambda_2\in (0,1)$ such that
	$\lambda_1\leq \lambda_2$ and $\Fix T_{\mathrm{DR}\lambda_2}\neq \emptyset$. 
	Moreover, let $T_{\mathrm{DR}\lambda_2}$ be single-valued at its fixed points. Then
	\begin{align}\label{eq:lam1lam2}
		P_B(\Fix T_{\mathrm{DR}\lambda_2})\subseteq P_B(\Fix T_{\mathrm{DR}\lambda_1}).
	\end{align}
	If \eqref{eq:fixed points} holds for $\lambda_2$ with equality instead of just set inclusion, then \eqref{eq:lam1lam2}
    holds with equality.
\end{propn}
\begin{proof}
	Let $x \in \Fix T_{\mathrm{DR}\lambda_2}\neq \emptyset$. Then, by \cref{cor: fixed point and its gap}, 
	we have the representation
	\begin{equation}
		x=P_B(x)-\frac{\lambda_2}{1-\lambda_2}\paren{P_B(x)-P_A(P_B(x))}.
	\end{equation}
	Consider $\xtilde\equiv P_B(x)-\frac{\lambda_1}{1-\lambda_1}\paren{P_B(x)-P_A(P_B(x))}$ and note, 
	as in the statements before, that $P_B(x)$ as well as $P_A(P_B(x))$ are single-valued, since 
	$x$ is a fixed point of $T_{\mathrm{DR}\lambda_2}\neq \emptyset$. Set $f\equiv P_B(x)$. Then
	$f\in B$ and $P_B(\xtilde)=f$. To see this, note that 
	$\frac{\lambda_2}{1-\lambda_2}\paren{P_A(f)-f}\in \pncone{B}(f)$. Since 
	$0\leq \tfrac{\lambda_1}{1-\lambda_1}\leq \tfrac{\lambda_2}{1-\lambda_2}$, $\tilde{x}$ is a convex combination of $f$ and $x$,
	from which we conclude that $P_B(\xtilde)=f$. Moreover, since $P_B(x)=f=P_B(\xtilde)$,
	we can conclude that $\xtilde \in \Fix T_{\mathrm{DR}\lambda_1} $. To see this, evaluate $T_{\mathrm{DR}\lambda_1}(\xtilde)$
	\begin{align*}
		T_{\mathrm{DR}\lambda_1}(\xtilde)=&\set{y}{y \in \lambda_1\paren{P_A(R_B(\xtilde))+\xtilde}+\paren{1-2\lambda_1}P_B(\xtilde)}\\
		=&\set{y}{y \in  \lambda_1\paren{P_A(2f-\xtilde)+\xtilde}+\paren{1-2\lambda_1}f},
	\end{align*}
	since $P_B(\xtilde)=\klam{f}$. $2f-\xtilde=2f-\paren{f-\tfrac{\lambda_1}{1-\lambda_1}\paren{f-P_A(f)}}$, where $P_A(f)$
	is single-valued since $x$ is a fixed point of $T_{\mathrm{DR}\lambda_2}$. This yields
	\begin{align*}
		2f-\xtilde=&f+\tfrac{\lambda_1}{1-\lambda_1}\paren{f-P_A(f)}\\
		=&P_A(f)+\tfrac{1}{1-\lambda_1}\paren{f-P_A(f)}.
	\end{align*}
	Analog to what we have seen before, we can argue that 
	$P_A(f)\in P_A(2f-\xtilde)$, 
	since $0\leq \tfrac{\lambda_1}{1-\lambda_1}\leq \tfrac{\lambda_2}{1-\lambda_2}$ and $P_A(f)=P_A(2f-x)=P_A(f+\tfrac{\lambda_2}{1-\lambda_2}\paren{f-P_A(f)})$.
	This implies that
	\begin{align*}
		&&\lambda_1\paren{P_A(f)+\xtilde}+\paren{1-2\lambda_1}f\in T_{\mathrm{DR}\lambda_1}\paren{\xtilde}\\
		&&\Longleftrightarrow\qquad\qquad\qquad\\
		&&\lambda_1\paren{P_A(f)+f-\tfrac{\lambda_1}{1-\lambda_1}\paren{f-P_A(f)}}+\paren{1-2\lambda_1}f\in T_{\mathrm{DR}\lambda_1}\paren{\xtilde}\\
		&&\Longleftrightarrow\qquad\qquad\qquad\\
		&&f-\tfrac{\lambda_1}{1-\lambda_1}\paren{f-P_A(f)}\in T_{\mathrm{DR}\lambda_1}\paren{\xtilde}\\
		&&\Longleftrightarrow \qquad\qquad\qquad\\
		&&\xtilde\in T_{\mathrm{DR}\lambda_1}\paren{\xtilde},
	\end{align*}
	and therefore $\xtilde \in \Fix T_{\mathrm{DR}\lambda_1}$. In conclusion,
	\begin{align*}
		P_B(\Fix T_{\mathrm{DR}\lambda_2})\subseteq P_B(\Fix T_{\mathrm{DR}\lambda_1}),
	\end{align*}
	which proves the claim.
\end{proof}	


\section{Quantitative Convergence Analysis}\label{s: conv analysis}
We proceed now to the main goal of our study, the convergence analysis of the algorithm.   
Almost all of the key properties of the relaxed Douglas-Rachford fixed point mapping, $\raar$, 
have been established in Section \ref{s: charac of fixed points}.  
The main idea for convergence goes back to Opial, \cite{Opial1967}. In our setting
nonemptiness of the fixed point set and averagedness of the mapping can be identified
as the essential properties yielding convergence of the iterative sequence.
It was shown in \cite{LukTebTha18}, however, 
that {\em gauge metric subregularity} of a fixed point mapping at its fixed points is a necessary 
condition for quantifiable (by said gauge) rates of convergence of the fixed point iteration.  
Metric subregularity is still missing from our development, and the 
main work of this section consists of deriving the conditions on the sets $A$ and $B$
under which (linear) metric subregularity holds.  

\begin{defn}[metric subregularity on a set]\label{def:2.17}
	Let $\Euclid$ and $\Ycal$ be Euclidean spaces, let $\mmap{\Phi}{\Euclid}{\Ycal}$,  and let $U \subset \Euclid, \ \ybar \in \Ycal$. 
   	The mapping $\Phi$ is called 
	\emph{metrically subregular for $\bar{y}$ on $U$ with constant $\kappa$ relative to $\Lambda\subset \Euclid$} 
	if
	\begin{align}\label{eq:45}
		\dist\paren{x, \Phi^{-1}(\ybar)\cap \Lambda}\leq \kappa \dist\paren{\ybar, \Phi(x)}
	\end{align}
	holds for all $x \in U \cap \Lambda$.  When $\ybar\in \Phi(\xbar)$, then  $\Phi$ is said to be 
	{\em metrically regular at $\xbar$ for $\ybar$} relative to $\Lambda$ when there exists a neighborhood $U$ of $\xbar$ and a constant $\kappa$ 
	such that \eqref{eq:45} holds for all $x \in U \cap \Lambda$.
		
	When $\Lambda=\Euclid $, the quantifier ``relative to''
	is dropped. The smallest constant $\kappa$ for which (\ref{eq:45}) holds is called 
	\emph{modulus} of metric subregularity.   
\end{defn}

The abstract result that allows us to quantify the convergence of $\raar$ follows.  
It is a simplified version of    
\cite[Corollary 2.3]{LukNguTam17} which
was later refined to show convergence to a specific point in \cite[Corollary 1]{LukTebTha18}.
The convergence result \cref{thm:2.18} is later specialized 
to $\raar$ in \cref{thm:conv raar} presented in \cref{s:lin conv of raar}.
\begin{thm}[(sub)linear convergence with metric subregularity]\label{thm:2.18}
	Let $\mmap T \Lambda \Lambda$ for $\Lambda\subset \Euclid$, with $\Fix T$ nonempty and closed, $\Phi\equiv T-\Id$.
	Denote $\paren{\Fix T +\delta \Ball}\cap \Lambda $ by $S_\delta$ for a nonnegative real $\delta$.		
	Suppose that, for all $\bar{\delta}>0$ small enough, there is a $\gamma\in (0,1)$, a nonnegative  
	scalar $\epsilon$ and a positive constant $\alpha$ bounded above 
	by $1$, such that,
	\begin{enumerate}[(i)]
		\item\label{ass:a} $T$ is pointwise almost averaged at all $y \in \Fix T \cap \Lambda$ with averaging constant 
		$\alpha$ and violation $\epsilon$ on $S_{\gamma\bar{\delta}}$, and
		\item\label{ass:b} for 
			\begin{align*}
				\bar{S}\equiv S_{\gamma\bar{\delta}}\setminus{\Fix T},
			\end{align*}
			$\Phi$ is metrically subregular for $0$ on $\bar{S}$ with constant $\kappa$ 
			relative to $\Lambda$. 
	  \end{enumerate}	 
	Then for any $x^0 \in \Lambda$ close enough to $\Fix T \cap \Lambda$, the iterates $x^{j+1}\in Tx^j$ satisfy
	\begin{align}\label{e:iterates}
		\dist\paren{x^{j+1}, \Fix T\cap \Lambda}\leq c\dist\paren{x^j, \Fix T \cap \Lambda}\quad \forall x^j \in \bar{S}
	\end{align}
        where $c\equiv \sqrt{1+\epsilon-\paren{\frac{1-\alpha}{\kappa^2\alpha}}}$.  If, in addition $\kappa$ satisfies
			\begin{align}\label{eq:kappa-epsilon}
				\kappa<\sqrt{\frac{1-\alpha}{\epsilon\alpha}}.
			\end{align}
	then, $\dist\paren{x^j, \tilde{x}}\rightarrow 0$ for some $\tilde{x}\in \Fix T\cap \Lambda$
	at least R-linearly with rate at most $c<1$. If $\Fix T\cap \Lambda$ is a single point, then
	convergence is Q-linear.
\end{thm}
We have already shown in \cref{thm: singlevalued at fixed points} that $\raar$ is almost averaged, 
{\em with any desired violation constant $\epsilon>0$}, 
at its fixed points on certain neighborhoods when $A$ and $B$ are super-regular at a distance. 
To achieve local linear convergence, inequality  \eqref{eq:kappa-epsilon} must hold, and this is where 
uniformity of almost averagedness with respect to $\epsilon$ is crucial:  as long as the mapping $\raar-\Id$, or 
a related mapping (see the discussion below), can be shown to  be relatively metrically subregular on 
a neighborhood  of $\Fix \raar$ - {\em regardless of the value of the modulus $\kappa$} - then suitable neighborhoods
can be found in the context of \cref{thm: singlevalued at fixed points} where the violation, 
$\epsilon$, is small enough that \eqref{eq:kappa-epsilon} is satisfied, and hence local linear convergence is 
guaranteed.  

The main work 
before us (Section \ref{s:raar at fix pts}) is to show metric subregularity of the appropriate mapping at points in the 
product space corresponding to fixed points of $\raar$. There are a number of ways to go about this, but 
all successful strategies we found are based on a characterization of the iterates on neighborhoods of fixed points 
lifted to a product space
where the tools are applied.  We were unable to provide a direct approach, involving the $\raar$ mapping itself, 
that guarantees metric subregularity from properties of the regularity of the sets $A$ and $B$ both individually (e.g. relative 
super-regularity at a distance) or as a collection (e.g. {\em subtransversality} discussed below).   
The characterization of 
the fixed points in \cref{thm:fixed points} allows us to build auxiliary {\em phantom} sets that are used in the 
analysis.  To adapt the framework above to the present setting 
we build a product space which represents not only the iterates of $\raar$ but also a 
cyclic projection between the phantom sets. In particular, we will define an operator in the product space
$\Euclid^4$ whose first entry is generated by applying $\raar$. The remaining three entries are generated 
by projecting the prior entry onto the sets $A$ and $B$ as well as phantom versions of these sets shifted
by a scaling of the local gap vector between $A$ and $B$ at the reference fixed point.

\subsection{\texorpdfstring{$\raar$}{T {DRlambda}} at \texorpdfstring{$\Fix \raar$}{Fix T {DRlambda}}: Metric Subregularity via Subtransversality}\label{s:raar at fix pts}

Direct verification of metric subregularity is notoriously difficult and verifying this for $\raar$ is no different.  
In principle, one must show that the coderivative (the generalized Jacobian) of the (multi-valued) $\raar$ mapping is 
injective on neighborhoods of $\Fix\raar$ \cite[Theorems~4B.1 and 4C.2]{DontchevRockafellar14}.  We were
unable to compute the coderivative of the $\raar$ mapping, let alone determine whether this is injective.  

Since our mapping is based on {\em projectors} to sets, another route is available for showing metric subregularity
which uses characterizations of the regularity of sets {\em in relation to one another}. In the context 
of consistent set feasibility, metric subregularity of a particular set-valued mapping on the product space has been shown to be  
equivalent to {\em subtransversality} of the {\em collection of sets} \cite{Kruger2018}.
\footnote{The terminology for this property in the literature is in disarray, and there are often several names 
with snappy prefixes for the same notion.}  This was expanded in \cite[Definition 3.2]{LukNguTam17} to 
account for {\em inconsistent set feasibility}.  Based on this more general notion of subtransversality of 
non-overlapping sets Luke et al.~were able to show that the cyclic projections mapping, 
$T_{CP}\equiv P_{\Omega_1}P_{\Omega_2}\cdots P_{\Omega_m}$ is metrically subregular 
when the collection of sets $\{\Omega_1, \dots, \Omega_m\}$ is subtransversal, and an additional 
technical assumption is satisfied.
We follow this approach here, but for the mapping $\raar$. Note that this general definition 
can simplify in special cases such as intersecting sets as is discussed in \cref{prop:lin reg}.
\begin{defn}[subtransversal collection of sets]\label{def:3.6}
	Let $\klam{\Omega_1, \Omega_2, \dots, \Omega_m}$ be a collection of nonempty closed subsets of $\Euclid$ and 
	define $\mmap \Psi {\Euclid^m} {\Euclid^m} $ by $\Psi(x)\equiv P_\Omega\paren{\Pi x }-\Pi x$ where 
	$\Omega\equiv \Omega_1\times \Omega_2\times \dots \times \Omega_m$, the projection $P_\Omega$ is with 
	respect to the Euclidean norm on $\Euclid^m$ and 
	$\Pi: x=\paren{x_1,x_2, \dots, x_m}\mapsto\paren{x_2, x_3,\dots, x_m,x_1}$ is
	the permutation mapping on the product space $\Euclid^m$ for $x_j \in \Euclid\ \paren{j=1,2, \dots, m}$. Let
	$\xbar=\paren{\xbar_1, \xbar_2, \dots, \xbar_m}\in \Euclid^m$ and $\ybar \in \Psi(\xbar)$.
 	The collection of sets is said to be \emph{subtransversal with constant $\kappa$ relative to 
	$\Lambda\subset \Euclid^m$ at $\xbar$ for $\ybar$} if $\Psi$ is metrically subregular  
	for $\ybar$ on some neighborhood $U$ of $\xbar$  
	with constant $\kappa$ relative to $\Lambda$.
\end{defn}	
In contrast to the original model setting, where $\klam{\Omega_1, \Omega_2, \dots, \Omega_m}$ is a 
collection of subsets on $\Euclid$, 
our definition of subtransversality is formulated on the product space $\Euclid^m$ where
$\Omega_1\times \Omega_2\times \cdots \times \Omega_m$ in $\Euclid^m$.   
\begin{lem}[subtransversality under addition]\label{lem: addition perserves subtransversality}
	Let $\klam{\Omega_1, \Omega_2 \dots, \Omega_m}\subset \Euclid$ be a subtransversal collection of sets 
	at a point $\xbar=(\xbar_1,\xbar_2,\dots,\xbar_m)$ for $\ybar\in \Psi(\xbar)$ relative to $\Lambda\subset \Euclid^m$ 
	with modulus $\kappa$. Then the collection 
	\[\klam{\Omega_1,\Omega_2 \dots, \Omega_m, \Omega_1-g, \Omega_2-g, \dots, \Omega_m-g }\subset \Euclid\]
	for some $g \in \Euclid$, is subtransversal at 
	\begin{align}\label{eq: xtilde}
		\tilde{x}=\paren{\xbar_1-g,\xbar_2, \xbar_3, \dots, \xbar_m,\xbar_1,\xbar_2-g,\xbar_3-g, \dots,\xbar_m-g}\in \Euclid^{2m}
	\end{align}
	 for 
	\[\tilde{y}=\paren{\ybar,\ybar}=\paren{\ybar_1,\ybar_2,\dots,\ybar_m,\ybar_1,\ybar_2,\dots,\ybar_m}\in \Euclid^{2m}\] 
	relative to
	\[\tilde{\Lambda}=\set{z\in \Euclid^{2m}}{\paren{z_{m+1},z_2,z_3,\dots, z_m}\in \Lambda, ~\paren{z_1,z_{m+2},z_{m+3},\dots, z_{2m}}\in \Lambda-\paren{g,g,\dots,g}}\]
	with modulus $\kappa$. 
\end{lem}
\begin{proof}
	We will show the result only for $m=2$ for reasons of simplicity and since one can easily enlarge 
	the number of sets used in the proof by the same pattern shown here. For $s\in \NN$ denote
	by $\Pi_{\Euclid}^s$ the permutation mapping on $\Euclid^s$.
	
	Let $U\subset \Euclid^2$ be a neighborhood of $\xbar\in \Euclid^2$ such that subtransversality 
	holds at $\xbar$ for $\ybar$ relative to $\Lambda$. Define $\Omega\equiv \Omega_1\times \Omega_2$ and therefore
	$\paren{\Omega_1-g}\times \paren{\Omega_2-g}=\Omega-\paren{g,g}$. Likewise set
	\[\tilde{U}\equiv \set{z\in \Euclid^4}{\paren{z_3,z_2}\in U, z_1=z_3-g, z_4=z_2-g}.\] 
	Thus every 
	$z \in \tilde{U}\cap \tilde{\Lambda}$ can be expressed as $\paren{x_1-g,x_2,x_1,x_2-g}^T$ for 
	some $\paren{x_1,x_2}\in U\cap \Lambda$.
	
	To show subtransversality of $\klam{\Omega_1,\Omega_2,\Omega_1-g,\Omega_2-g}$ we have to verify metric 
	subregularity of $\Psi=P_\Omega(\Pi^4_\Euclid )-\Pi^4_\Euclid $	
	for $\tilde{y}\in \Psi(\tilde{x})$ relative to $\tilde{\Lambda}$ on $\tilde{U}$, a neighborhood of $\tilde{x}$. 

	First, we show that $\tilde{y}\in \Psi(\tilde{x})$, i.e.  
	$\tilde{y}\in P_{\Omega\times {\Omega - (0,0,g, g)}}\paren{\Pi^4_\Euclid (\tilde{x})}-\Pi^4_\Euclid(\tilde{x})$.
	Let $\tilde{x}$ be defined by \eqref{eq: xtilde} then
	\begin{align}
		&P_{\Omega\times {\Omega - (0,0,g, g)}}\paren{\Pi^4_\Euclid (\tilde{x})}-\Pi^4_\Euclid(\tilde{x})\nonumber\\
		&\qquad\qquad=\paren{P_\Omega\paren{\tilde{x}_2,\tilde{x}_3}-\paren{\tilde{x}_2,\tilde{x}_3},P_{\Omega-\paren{g,g}}\paren{\tilde{x}_4,\tilde{x}_1}-\paren{\tilde{x}_4,\tilde{x}_1}}\nonumber\\
		&\qquad\qquad=\paren{P_\Omega\paren{\xbar_2,\xbar_1}-\paren{\xbar_2,\xbar_1},P_{\Omega-\paren{g,g}}\paren{\xbar_2-g,\xbar_1-g}-\paren{\xbar_2-g,\xbar_1-g}}\nonumber\\
		&\qquad\qquad=\paren{P_\Omega\paren{\xbar_2,\xbar_1}-\paren{\xbar_2,\xbar_1},P_{\Omega}\paren{\xbar_2,\xbar_1}-\paren{\xbar_2,\xbar_1}},\label{eq: 3rd step}
	\end{align}
	where the last equality holds since $P_{C-g}(x-g)=P_C(x)-g$ for any set $C$. Then \eqref{eq: 3rd step} yields
	\begin{align*}
		&P_{\Omega\times {\Omega - (0,0,g, g)}}\paren{\Pi^4_\Euclid (\tilde{x})}-\Pi^4_\Euclid(\tilde{x})\\
		&\qquad\qquad=\paren{P_\Omega\paren{\xbar_2,\xbar_1}-\paren{\xbar_2,\xbar_1},P_{\Omega}\paren{\xbar_2,\xbar_1}-\paren{\xbar_2,\xbar_1}}\\
		&\qquad\qquad	\ni\paren{\ybar,\ybar}=\tilde{y},
	\end{align*}
	since $\ybar \in P_\Omega(\Pi^2_\Euclid\xbar)-\Pi^2_\Euclid\xbar$ by the assumptions on 
	subtransversality of $\klam{\Omega_1,\Omega_2, \dots, \Omega_m}$. By $\tilde{x}\in \tilde{\Lambda}$
	this shows $\tilde{y}\in \Psi(\tilde{x})$ as claimed.
	
	It remains to prove that inequality \eqref{eq:45} holds for $\Psi$ and at $\tilde{x}$
	for $\tilde{y}\in \Psi(\tilde{x})$ relative to $\tilde{\Lambda}$ on $\tilde{U}$.
	For this, take a $\paren{x_1-g,x_2,x_1,x_2-g}^T\in \tilde{U}\cap \tilde{\Lambda}$, then:
	\scriptsize{
	\begin{align}
		&\kappa^2 \dist^2\paren{P_{\Omega\times {\Omega - (0,0,g, g)}}\paren{\Pi_\Euclid^4\paren{\begin{array}{c}x_1-g\\x_2\\x_1\\x_2-g\end{array}}}-\Pi_\Euclid^4\paren{\begin{array}{c}x_1-g\\x_2\\x_1\\x_2-g\end{array}},\paren{\begin{array}{c}\ybar_1\\\ybar_2\\\ybar_1\\\ybar_2\end{array}}}\\
		=&\kappa^2\paren{\dist^2\paren{P_\Omega\paren{\begin{array}{c}x_2\\x_1\end{array}}-\paren{\begin{array}{c}x_2\\x_1\end{array}},\paren{\begin{array}{c}\ybar_1\\\ybar_2\end{array}}}+\dist^2\paren{P_{\Omega-\paren{g,g}}\paren{\begin{array}{c}x_2-g\\x_1-g\end{array}}-\paren{\begin{array}{c}x_2-g\\x_1-g\end{array}},\paren{\begin{array}{c}\ybar_1\\\ybar_2\end{array}}}},\label{eq: e4 in e2}
	\end{align}
	}
	\normalsize
	by rewriting the distance on $\Euclid^4$ in terms of the distance on $\Euclid^2$. Using 
	again that $P_{C-g}(x-g)=P_C(x)-g$ for an arbitrary set $C$, 
	\eqref{eq: e4 in e2} ends up as
	\scriptsize{
	\begin{align}
		&\kappa^2\paren{\dist^2\paren{P_\Omega\paren{\Pi_\Euclid^2\paren{\begin{array}{c}x_1\\x_2\end{array}}}-\Pi_\Euclid^2\paren{\begin{array}{c}x_1\\x_2\end{array}},\paren{\begin{array}{c}\ybar_1\\\ybar_2\end{array}}}+\dist^2\paren{P_{\Omega}\paren{\begin{array}{c}x_2\\x_1\end{array}}-\paren{\begin{array}{c}x_2\\x_1\end{array}},\paren{\begin{array}{c}\ybar_1\\\ybar_2\end{array}}}}\nonumber\\
				=&2\kappa^2\paren{\dist^2\paren{P_\Omega\paren{\begin{array}{c}x_2\\x_1\end{array}}-\paren{\begin{array}{c}x_2\\x_1\end{array}},\paren{\begin{array}{c}\ybar_1\\\ybar_2\end{array}}}}\\
		\geq& 2\dist^2\paren{\paren{\begin{array}{c}x_1\\x_2\end{array}},\paren{P_\Omega\paren{\Pi_\Euclid^2 \paren{\cdot}}-\Pi_\Euclid^2\paren{\cdot}}^{-1}\paren{\begin{array}{c}\ybar_1\\\ybar_2\end{array}}},\label{eq: using subtrans}
	\end{align}
	}
	\normalsize
	where the last inequality holds by subtransversality of $\klam{\Omega_1,\Omega_2}$ at $(\xbar_1,\xbar_2)$ for $(\ybar_1,\ybar_2)$ relative to
	$\Lambda$ with modulus $\kappa $ on $U$. Rewriting \eqref{eq: using subtrans} in the distance on $\Euclid^4$ yields
	\scriptsize{
	\begin{align*}		&\kappa^2 \dist^2\paren{P_{\Omega\times {\Omega - (0,0,g, g)}}\paren{\Pi_\Euclid^4\paren{\begin{array}{c}x_1-g\\x_2\\x_1\\x_2-g\end{array}}}-\Pi_\Euclid^4\paren{\begin{array}{c}x_1-g\\x_2\\x_1\\x_2-g\end{array}},\paren{\begin{array}{c}\ybar_1\\\ybar_2\\\ybar_1\\\ybar_2\end{array}}}\nonumber \\
		\geq&\dist^2\paren{\paren{\begin{array}{c}x_1\\x_2\\x_1\\x_2\end{array}},\set{\paren{\begin{array}{c}z_1\\z_2\\z_3\\z_4\end{array}}}{P_\Omega\paren{\begin{array}{c}z_2\\z_1\end{array}}-\paren{\begin{array}{c}z_2\\z_1\end{array}}\ni \paren{\begin{array}{c}\ybar_1\\\ybar_2\end{array}},P_{\Omega-\paren{g,g}}\paren{\begin{array}{c}z_4\\z_3\end{array}}-\paren{\begin{array}{c}z_4\\z_3\end{array}}\ni \paren{\begin{array}{c}\ybar_1\\\ybar_2\end{array}}}}\\
		=&\dist^2\paren{\paren{\begin{array}{c}x_1\\x_2\\x_1\\x_2\end{array}},\set{\paren{\begin{array}{c}z_1\\z_2\\z_3\\z_4\end{array}}}{P_\Omega\paren{\begin{array}{c}z_2\\z_1\end{array}}-\paren{\begin{array}{c}z_2\\z_1\end{array}}\ni \paren{\begin{array}{c}\ybar_1\\\ybar_2\end{array}},P_{\Omega-\paren{g,g}}\paren{\begin{array}{c}z_4\\z_3\end{array}}-\paren{\begin{array}{c}z_4\\z_3\end{array}}\ni \paren{\begin{array}{c}\ybar_1\\\ybar_2\end{array}}}}\\
		=&\dist^2\paren{\paren{\begin{array}{c}x_1\\x_2\\x_1-g\\x_2-g\end{array}},\set{\paren{\begin{array}{c}z_1\\z_2\\z_3\\z_4\end{array}}}{P_{\Omega\times {\Omega - (0,0,g, g)}}\paren{\begin{array}{c}z_2\\z_1\\z_4\\z_3\end{array}}-\paren{\begin{array}{c}z_2\\z_1\\z_4\\z_3\end{array}}\ni \paren{\begin{array}{c}\ybar_1\\\ybar_2\\\ybar_1\\\ybar_2\end{array}}}}\\
		=&\dist^2\paren{\paren{\begin{array}{c}x_1-g\\x_2\\x_1\\x_2-g\end{array}},\set{\paren{\begin{array}{c}z_3\\z_2\\z_1\\z_4\end{array}}}{P_{\Omega\times {\Omega - (0,0,g, g)}}\paren{\begin{array}{c}z_2\\z_1\\z_4\\z_3\end{array}}-\paren{\begin{array}{c}z_2\\z_1\\z_4\\z_3\end{array}}\ni \paren{\begin{array}{c}\ybar_1\\\ybar_2\\\ybar_1\\\ybar_2\end{array}}}}\\
		=&\dist^2\paren{\paren{\begin{array}{c}x_1-g\\x_2\\x_1\\x_2-g\end{array}},\paren{P_{\Omega\times {\Omega - (0,0,g, g)}}\paren{\Pi_\Euclid^4\paren{\cdot}}-\Pi_\Euclid^4\paren{\cdot}}^{-1}\paren{\begin{array}{c}\ybar_1\\\ybar_2\\\ybar_1\\\ybar_2\end{array}}},
	\end{align*}
	}
	\normalsize 
	where the last three steps were just rearranging the expression to get the claimed result.
	This completes the proof.	
\end{proof}
\begin{rem}\label{rem:capser}
	The points involved in \cref{lem: addition perserves subtransversality} change if we change the order of 
	the sets involved. Of particular interest for our later analysis is the case of two sets $A$ and $B$
	where we change the order on the product space in the following way
	\begin{equation}\label{eq:casper}
		(B-g)\times (A-g)\times A\times B\subset \Euclid^4,
	\end{equation}
	in contrast to the order $A\times B \times (A-g)\times (B-g)$ as used in \cref{lem: addition perserves subtransversality}.
	Therefore the points $\xtilde$ and $\ytilde$ as well as the set $\tilde{\Lambda}$ change to
	\begin{align*}
		\xtilde'=&\paren{\xbar_1,\xbar_1-g,\xbar_2-g,\xbar_2}\\
		\ytilde'=&\paren{-\ybar_1,-\ybar_2,\ybar_1,\ybar_2}\\
		\tilde{\Lambda}'=&\set{z\in \Euclid^4}{\paren{z_3,z_4}\in \Lambda, ~\paren{z_2,z_1}\in \Lambda-\paren{g,g}}.
	\end{align*}
	That is, the collection $\klam{B-g,A-g,A,B}\subset \Euclid$ is subtransversal at $\tilde{x}'$ for $\tilde{y}'$
	relative to $\tilde{\Lambda}'$. Note that the negative part of $\ytilde$ emerged from the changed order of $B$ and $A$
	in comparison to \cref{lem: addition perserves subtransversality}.
\end{rem}
We are now ready to construct the product space on which we determine metric subregularity via subtransversality. 	
Instead of the two original sets, we consider four sets:  the sets $A, B$ and shifted sets 
$B-\frac{\lambda}{1-\lambda}g$ and $A-\frac{\lambda}{1-\lambda}g$ for some gap vector $g$. 
Our aim is to show local linear convergence of $\raar$
by adapting the approach developed in \cite{LukNguTam17} for cyclic projections.
There it was essential that one of the sets involved contains the fixed points of the mapping.
The reason for including the set $B-\frac{\lambda}{1-\lambda}g$ in our problem, therefore, lies in the
characterization of the fixed point set of the $\raar$ mapping. As established in
\cref{thm:fixed points} and \cref{cor: fixed point and its gap} fixed points $\xbar$ of $\raar$
at which $\raar$ is single-valued can be described as
\[\klam{\xbar}=P_B(\xbar)-\frac{\lambda}{1-\lambda}\paren{P_B(\xbar)-P_A(P_B(\xbar))},\]
which is an element in $B-\frac{\lambda}{1-\lambda}g$ when $g=P_B(\xbar)-P_A(P_B(\xbar))$. 
Thus, locally $B-\frac{\lambda}{1-\lambda}g$ contains
fixed points of $\raar$. 
To be able to apply our results established in \cref{lem: addition perserves subtransversality}
we have to consider the set $A-\frac{\lambda}{1-\lambda}g$ as well.

We  denote by $\Omega_{g}$ the product of the collection of sets 
$\klam{B-\frac{\lambda}{1-\lambda}g, A-\frac{\lambda}{1-\lambda}g, A , B}$. That is,
\begin{align*}
	\Omega_{g}\equiv\paren{B-\frac{\lambda}{1-\lambda}g}\times \paren{A-\frac{\lambda}{1-\lambda}g}\times A \times B.
\end{align*}
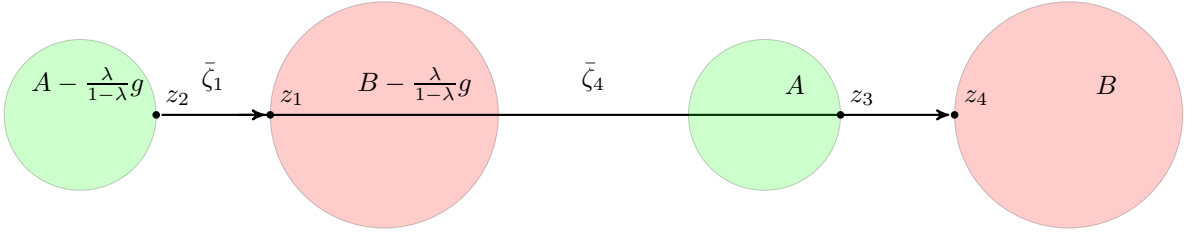
\begin{figure}[ht]
	\centering
	\begin{tikzpicture}
			\draw[fill=green, opacity=0.2] (0,0) circle (1cm);
			\draw (0.1,0.4) node {$A-\frac{\lambda}{1-\lambda}g$};
			\draw[fill=red, opacity=0.2] (4,0) circle (1.5cm);
			\draw (4.4,0.4) node {$B-\frac{\lambda}{1-\lambda}g$};
			\draw[thick, ->, >=stealth', shorten <=2pt, shorten
    >=2pt] (1,0) -- (2.5,0);
			\draw (1.75,0.5) node {$\bar{\zeta}_1$};
			\draw[thick, ->, >=stealth', shorten <=2pt, shorten
    >=2pt] (2,0) -- (11.5,0);			
			\draw (6.75,0.5) node {$\bar{\zeta}_4$};
			\draw[fill=green, opacity=0.2] (9,0) circle (1cm);
			\draw (9.4,0.4) node {$A$};
			\draw[fill=red, opacity=0.2] (13,0) circle (1.5cm);
			\draw (13.5,0.4) node {$B$};
			\fill (1,0) circle (0.05) node[above right] {$z_2$};
			\fill (2.5,0) circle (0.05) node[above right] {$z_1$};
			\fill (10,0) circle (0.05) node[above right] {$z_3$};
			\fill (11.5,0) circle (0.05) node[above right] {$z_4$};
	\end{tikzpicture}
	\caption{Framework for the convergence analysis illustrated in $\Euclid$.}\label{fig:framework}
\end{figure}
Define 
\begin{equation}\label{eq:W0}
    \begin{split}
		W_0(g) \equiv & \left\{u=(u_1,u_2,u_3,u_4) \in \Euclid^4\,\mid\, \right. \\
		& \left. u_1\in P_{B-\frac{\lambda}{1-\lambda}g}(u_2),\ u_2 \in P_{A-\frac{\lambda}{1-\lambda}g}(u_3), \ u_3\in P_A(u_4), \ u_4\in P_B(u_1)\right\}.
    \end{split}
\end{equation}
This is the set of fixed points of the mapping $P_{\Omega_{g}}\circ\Pi$ in the product space $\Euclid^4$ 
corresponding to a cycle of the cyclic 
projections operator $P_{B-\frac{\lambda}{1-\lambda}g}P_{A-\frac{\lambda}{1-\lambda}g}P_AP_B$.
By our construction, the set $W_0(g)$ could be (and for generic $g$ {\em will be}) empty;  this would be the case 
when $g$ does not correspond to a difference vector.  
The set of {\em difference vectors, $\zeta$,}  is denoted by $\mathcal{Z}(x, g)$ and defined by 
\begin{align}\label{eq:Z(u,g)}
	\mathcal{Z}(x, g)&\equiv\set{\zeta\equiv z-\Pi z }{ z \in W_0(g)\subset\Euclid^4, \ z_1=x}.
\end{align}
The last set to introduce is 
\begin{align}\label{eq:W(zeta)}
	W\paren{\zetabar}\equiv \set{u \in \Euclid^4}{ u-\Pi u =\zetabar}.
\end{align}
This set is an affine transformation of the diagonal of the product space 
and serves as a characterization of the local geometry of the sets in relation to each other at fixed 
points of $\raar$.

These sets, of course, only make sense in the context of local nearest points between the components. 
In particular, we are interested in points $x\in \Euclid$ associated with fixed points of $\raar$ and 
their associated shadow points and gap vectors, respectively $b \in P_B(x)$ and $g\in b-P_A(b)$ 
(the local gap between $A$ and $B$). 
Note that by \cref{thm: singlevalued at fixed points} 
for fixed points of $\raar$ at which $\raar$ is single-valued we have $\{b\}= P_B(x)$ and the gap 
vector $g$ is unique. 
When $x$ is a fixed point, the set $\Zcal(x,g)$ characterizes the distance between the
cyclically projected iterates of $\raar$ on the individual sets. This enables us to distinguish different 
fixed points of $\raar$ according to their respective difference vectors. 

Let the assumptions of \cref{thm: singlevalued at fixed points} hold at $\ubar\in \Fix\raar$ and let 
$z \in W_0(g)\subset\Euclid^4$ for $z_1=\ubar$ and $\{g\}=P_B(\ubar)-P_AP_B(\ubar)$. Then 
\cref{thm: singlevalued at fixed points} yields
\begin{enumerate}[(i)]
	\item $\begin{aligned}[t]
    z_3-z_4 \in&~ P_A\paren{P_B(z_1)}-P_B(z_1)=\{-g\}
  \end{aligned}$
  \item\label{item:zwei} $\begin{aligned}[t]
    z_4-z_1\in&~P_B(z_1)-(P_B(z_1)-\tfrac{\lambda}{1-\lambda}g)=\klam{\tfrac{\lambda}{1-\lambda}g},
  \end{aligned}$
\end{enumerate}
where \eqref{item:zwei} holds by \cref{thm:fixed points} which is applicable since $\raar$ is single-valued at $\ubar$
by \cref{thm: singlevalued at fixed points}. Moreover, we get by the assumptions of \cref{thm: singlevalued at fixed points}
that $P_A(R_B(\ubar))=P_A(P_B(\ubar))=z_3$. Since $z_3+\tfrac{\lambda}{1-\lambda}g$ lies in a straight line between $R_B(\ubar)=2P_B(\ubar)-
\ubar=2P_B(\ubar)-P_B(\ubar)+\tfrac{\lambda}{1-\lambda}g=P_A(P_B(\ubar))+(1+\tfrac{\lambda}{1-\lambda})g$ and $P_A(R_B(\ubar))$ 
we also deduce $P_A(z_3+\tfrac{\lambda}{1-\lambda}g)=z_3$. Using again \cref{thm:fixed points} yields
\begin{enumerate}[(i)]
	\setcounter{enumi}{2}
	\item $\begin{aligned}[t]
    z_1-z_2\in &~ P_B(\ubar)-\tfrac{\lambda}{1-\lambda}g-P_{A-\tfrac{\lambda}{1-\lambda}g}(z_3)\\
    =&~ P_B(\ubar)-\tfrac{\lambda}{1-\lambda}g-P_{A}(z_3+\tfrac{\lambda}{1-\lambda}g)+\tfrac{\lambda}{1-\lambda}g\\
    =&~ P_B(\ubar)-z_3\\
    =&~ P_B(\ubar)-P_A(P_B(\ubar))\\
    =&~ \klam{g};\\
  \end{aligned}$
	\item $\begin{aligned}[t]
  	 z_2-z_3=&~(z_2-z_1)+(z_1-z_4)+(z_4-z_3)=-\tfrac{\lambda}{1-\lambda}g.
 	 \end{aligned}$
\end{enumerate}
Figure 
\ref{fig:framework} illustrates the sets and difference vectors above. 
The individual entries of $z$ relate to the cyclically projected fixed 
point $x$ on each of the individual sets.
		
Along with the definitions above we define the operator
\begin{align*}
	\mmap {T_{\zetabar}} {\Euclid^4}{\Euclid^4}: u \mapsto 
	\set{\paren{u_1^+, u_1^+ -\zetabar_1,u_1^+ -\zetabar_1- \zetabar_2, u_1^+-\zetabar_1-\zetabar_2-\zetabar_3 }}{ \ u_1^+\in \raar u_1},
\end{align*}
for $\zetabar\in \mathcal{Z}(\xbar, g)$ where $\xbar\in \Fix \raar$ and $g=P_B(\xbar)-P_AP_B(\xbar)$. Note that 
$0=\zetabar_1+\zetabar_2+\zetabar_3+\zetabar_4$, so the expression above can be simplified to 
\begin{align}\label{eq:Tzeta}
	\mmap {T_{\zetabar}} {\Euclid^4}{\Euclid^4}: u \mapsto 
	\set{\paren{u_1^+, u_1^+ -\zetabar_1,u_1^+ -\zetabar_1- \zetabar_2, u_1^++\zetabar_4}}{ \ u_1^+\in \raar u_1}.
\end{align} 
The  mapping $T_{\zetabar}$ represents the iterates of $\raar$ on the space $\Euclid$ by shifting each iterate by
some fixed difference vector $\zetabar$.  We presume, in what follows, that $\zetabar$ is the difference 
vector corresponding to the fixed point to which our iteration is converging.  Of course, when one does not 
know the location of the fixed points, it is unlikely that the corresponding difference vector will be known, but 
this situation is no different than other studies which assume that the problem is consistent, and that all fixed 
points correspond to the zero difference vector. Our aim here is not to determine the difference vector or the fixed 
point, but rather to provide a quantification of the convergence based on verifiable regularity of the fixed point 
mapping {\em in neighborhoods} of fixed points.  

We are now ready to start building our argument.  
The following lemma establishes a connection between fixed points of $T_{\zetabar}$
to fixed points of $\raar$. 
\begin{lem}\label{prep metric subreg}
	Let $\lambda \in (0,1)$ and $A,B \subset \Euclid$ both nonempty and closed.
	Fix $\xbar \in \Fix \raar \neq \emptyset$ with $\raar$ being single-valued at $\xbar$
	and set $g\equiv P_B(\xbar)-P_A(P_B(\xbar))$.
	Furthermore, let $\zetabar\in \Zcal(\xbar,g)$ and define 
	$\Psi_{g}\equiv \paren{P_{\Omega_{g}}}\circ \Pi - \Pi $ 
	as well as $\Phi_{\zetabar}\equiv T_{\zetabar}-\Id$. Then the following hold.
	\begin{enumerate}[(i)]
		\item\label{item:1} $T_{\zetabar}$ maps $W\paren{\zetabar}$ to itself. Moreover $u \in \Fix T_{\zetabar}$
					if and only if $u \in W\paren{\zetabar}$ with $u_1\in \Fix \raar$. 
		\item\label{item:2} 
		\begin{align*}
			\Psi_{g}^{-1}\paren{\zetabar}\cap W\paren{\zetabar}\cap \mathcal{N}\subseteq \Phi_{\zetabar}^{-1}\paren{0}\cap W\paren{\zetabar},
		\end{align*}
		where $\mathcal{N}\equiv \set{z \in \Euclid^4}{P_A(z_4+\tfrac{\lambda}{1-\lambda}g)=z_3}$.
		\item\label{item:3} If the distance is with respect to the Euclidean norm, then 
		$\dist \paren{0,\Phi_{\zetabar}(u)}=2\dist\paren{u_1, \raar u_1}$ for $u \in W(\zetabar)$.
	\end{enumerate}	 
\end{lem}

\begin{rem}\label{e:Ncal}
Note that the set $\mathcal{N}$ guarantees equality of the description of the fixed point set in \cref{thm:fixed points}.
In our main result $\mathcal{N}$ will not appear anymore. This is due to the fact that the assumptions of \cref{thm: singlevalued at fixed points}
assure that the neighborhood we consider will be a subset of $\mathcal{N}$.
\end{rem}
\begin{proof}[Proof of \cref{prep metric subreg}]
	\eqref{item:1}. The first part of \eqref{item:1} follows immediately by the definition of $T_{\zetabar}$ 
	and $W(\zetabar)$. Now let $u \in \Fix T_{\zetabar}$,
	\begin{align*}
		&&\Longleftrightarrow&& &u_1 \in \Fix \raar \text{ and } u_2=u_1 -\zetabar_1,~ u_3=u_1 -\zetabar_1- \zetabar_2, ~u_4=u_1+\zetabar_4\\
		&&\Longleftrightarrow&& &u_1 \in \Fix \raar \text{ and } u_2=u_1 -\zetabar_1, ~u_3=u_2- \zetabar_2, ~u_4=u_1+\zetabar_4\\
		&&\Longleftrightarrow&& &u_1 \in \Fix \raar \text{ and } u \in W(\zetabar),
	\end{align*}
	which proves the rest of \eqref{item:1}.
			
	\eqref{item:2}. For the second part of the lemma let 
	$z \in \Psi^{-1}_{g}\paren{\zetabar}\cap W(\zetabar)\cap \mathcal{N}$. This means nothing more than 
	\begin{align*}
		\zetabar\in \Psi_{g}\paren{z} \text{ and }z-\Pi z =\zetabar,
	\end{align*} 
	which is equivalent to
	\begin{align*}
		\zetabar\in P_{\Omega_{g}}\Pi z -\Pi z \text{ and }z-\Pi z =\zetabar.
	\end{align*}
	This implies
	\begin{align*}
		z_1 \in P_{B-\frac{\lambda}{1-\lambda}g}P_{A-\frac{\lambda}{1-\lambda}g}P_AP_Bz_1 \text{ and } z-\Pi z=\zetabar.
	\end{align*}		
	The mapping $\Phi_{\zetabar}(z)=T_{\zetabar}z -z$ has the image $\paren{0,0}$ if $z_1 \in \Fix \raar z_1$.
	By $\zetabar_4=z_4-z_1$ and $\zetabar_4\in P_B(z_1)-z_1=\tfrac{\lambda}{1-\lambda}$ we know that $z_4\in P_B(z_1)$.
	This together with the definition of $\mathcal{N}$ yields 
	$P_A\paren{R_B(z_1)}\ni P_A\paren{2z_4-z_1}=P_A\paren{z_4+\tfrac{\lambda}{1-\lambda}g}=z_3$. Inserting this 
	in $\raar(z_1)$ yields
	\begin{align*}
		\raar(z_1) = &\lambda \paren{P_A\paren{R_B(z_1)}+z_1}+\paren{1-2\lambda}P_B(z_1)\\
		\ni & \lambda \paren{z_3+z_1}+\paren{1-2\lambda}z_4\\
		= & z_1+\lambda(z_3-z_4)+\paren{1-\lambda}(z_4-z_1)\\
		 = & z_1+\zetabar_3+\paren{1-\lambda}\zetabar_4\\
		 = & z_1,
	\end{align*}
	since $\zetabar$ is generated by a fixed point of $\raar$. Thus $z_1\in \Fix \raar$, 
	which proves $z\in \Phi_{\zetabar}^{-1}\paren{0}$ and 
	completes the proof of \eqref{item:2}.
		 
	\eqref{item:3}. This part of the proof is a routine calculation:
	\begin{align*}
		&\dist \paren{0,\Phi_{\zetabar}(u)}\\
		&=\dist \paren{0,T_{\zetabar}u-u}\\
		&=\sqrt{\dist^2\paren{0,\raar u_1 -u_1}+\dist^2\paren{0,\raar u_1-\zetabar_1 -u_2}+\cdots +\dist^2\paren{0,\raar u_1 -\sum_{i=1}^3 \zetabar_i-u_4}}\\
		&=\sqrt{4 \dist^2\paren{0,\raar u_1 -u_1}}\\
		&= 2\dist\paren{0, \raar u_1-u_1}.
	\end{align*}
\end{proof}
We are now ready for the main result of this subsection.  We show that the mapping $T_\zetabar-\Id$ is metrically 
subregular on neighborhoods of its zeros;  from this we can conclude that the fixed point iteration generated by the 
mapping $T_\zetabar$
is locally linearly convergent, from which we will be able to deduce local linear convergence of $\raar$. 
\begin{propn}[metric subregularity of $T_\zetabar$ by subtransversality]\label{prop:metric subreg}
	Let $\lambda \in (0,1)$, $\xbar \in \Fix \raar$  with $\raar$ being single-valued at 
	$\xbar$ and set $g\equiv P_B(\xbar)-P_A(P_B(\xbar))$.
	Furthermore, let $\zetabar\in \mathcal{Z}\paren{\xbar, g}$ and 
	$\ubar=\paren{\ubar_1,\ubar_2, \ubar_3, \ubar_4}\in W_0(g)$ satisfy $\zetabar=\ubar-\Pi \ubar$ with 
	$\ubar_1=\xbar$. 
    Let $T_\zetabar$ be defined by \eqref{eq:Tzeta} and define $\Phi_\zetabar\equiv T_\zetabar-\Id$. 
	Suppose the following hold:
	\begin{enumerate}[(i)]
		\item\label{item: subtransv} the collection of sets $\klam{B-\tfrac{\lambda}{1-\lambda}g, A-\tfrac{\lambda}{1-\lambda}g,A,B}$ 
		is subtransversal at $\ubar$ for $\zetabar$ relative to $\Lambda \subseteq W\paren{\zetabar}$ 
		with constant $\kappa$ and neighborhood $U$ of $\ubar$;
		\item\label{item: technical ass} there exists a positive constant $\sigma $ such that
			\[\dist \paren{\zetabar, \Psi_{g}(u)}\leq \sigma \dist\paren{0, \Phi_{\zetabar}(u)}, \quad \forall u \in \Lambda\cap U \text{ with }u_1 \in B-\tfrac{\lambda}{1-\lambda}g.\]
	\end{enumerate}
	Then the mapping $\Phi_{\zetabar}\equiv T_{\zetabar}-\Id $ is metrically subregular for 0 on U
	relative to $\Lambda \cap \mathcal{N}$ with constant $\bar{\kappa}=\kappa\sigma$,
	where $\mathcal{N}\equiv \set{z \in \Euclid^4}{P_A(z_4+\tfrac{\lambda}{1-\lambda}g)=z_3}$.
\end{propn}
\begin{proof}
	This is an application of the assumptions and \cref{prep metric subreg}\eqref{item:2}
	\begin{align*}
		\paren{\forall u \in U \cap \Lambda \cap \mathcal{N}\text{ with }u_1\in B-\tfrac{\lambda}{1-\lambda}g} \quad \dist\paren{u, \Phi_{\zetabar}^{-1}(0)\cap \Lambda \cap \mathcal{N}}&\leq \dist \paren{u, \Psi_{g}^{-1}(0)\cap \Lambda \cap \mathcal{N}}\\
		&\leq \kappa \dist\paren{\zetabar,\Psi_{g}(u)}\\
		&\leq \kappa \sigma \dist\paren{0, \Phi_{\zetabar}(u)},
	\end{align*}
	i.e. $\Phi$ is metrically subregular for $0$ on $U$ relative to $\Lambda \cap \mathcal{N}$ with constant 
	$\bar{\kappa}$, as claimed.
\end{proof}		
By \cref{thm:2.18},\cref{prop:metric subreg} and \cref{thm: singlevalued at fixed points}
the {\em three} ingredients to get convergence are given by the regularity of the sets $A$ and $B$, 
subtransversality of the collection of sets $\klam{A,B}$ and the additional 
assumption \eqref{item: technical ass} in \cref{prop:metric subreg}. As seen 
in \cite[Proposition 3.5]{LukNguTam17} this is also true
for the alternating projection algorithm. If the intersection $A\cap B$ is nonempty, assuming the 
stronger property of transversality, super-regularity is enough to show convergence of the
Douglas-Rachford algorithm, \cite[Theorem 6.8]{Phan16}\cite[Theorem 3.18]{HesseLuk2013}.  For alternating projections 
one only needs transversality at points of intersection and super-regularity of {\em one} 
of the sets \cite[Theorem 5.16]{LewLukMal2009}.  In any case,  the additional assumption 
\eqref{item: technical ass}, is not needed when the assumptions on the fixed points are strong enough.  
This is also the case for consistent feasibility and the relaxed Douglas-Rachford method
as seen next.
\begin{propn}[intersecting sets]\label{propn:technical ass for consistent}
	As before let $\lambda \in (0,1)$. Moreover, assume that the intersection of $A$ and $B$ is 
	nonempty, i.e. $A\cap B\neq \emptyset$.	Thus, for every $\xbar \in A\cap B\subset \Fix \raar$
	we have $g\equiv P_B(\xbar)-P_A(P_B(\xbar))=0$.
	Furthermore, let $\zetabar\in \mathcal{Z}\paren{\xbar, g}$. 
    Then \eqref{item: technical ass} in \cref{prop:metric subreg} is always satisfied on  
   	$\Lambda\subset W(\zetabar)$ with $\sigma=\frac{1}{\sqrt{2}\lambda}$.
\end{propn}
\begin{proof}
	Since $\xbar\in A\cap B$ and $g=0$, we get $\zetabar=(0,0,0,0)$.
	Moreover, note that for every $b\in B$ we 
	gather $\raar(b)-b=\lambda(P_A(b)-b)$, since
	\begin{align*}
		\raar(b)-b
			&= \frac{\lambda}{2}\paren{R_AR_B(b)+b}+(1-\lambda)P_B(b)-b\\
			&= \frac{\lambda}{2}\paren{R_A(b)+b}+(1-\lambda)b-b\\
			&= \lambda P_A(b)-\lambda b\\
			&= \lambda\paren{ P_A(b)- b}.					
	\end{align*}
	Therefore, we deduce for 
	$u \in \Lambda \subset W(\zetabar)=\set{u\in \Euclid^4}{u_i=u_j,~i,j\in \klam{1,2,3,4}}$  
	with $u_1\in B$
	\begin{align*}
		T_0(u)-u
			&=\paren{\raar(u_1)-u_1,\raar(u_1)-u_1,\raar(u_1)-u_1,\raar(u_1)-u_1}\\
			&=\paren{\lambda\paren{ P_A(u_1)- u_1},\lambda\paren{ P_A(u_1)- u_1},\lambda\paren{ P_A(u_1)- u_1},\lambda\paren{ P_A(u_1)- u_1}},
	\end{align*}
	and thus
	\begin{align}\label{eq:righthandside}
		\dist^2\paren{0, \Phi_{\zetabar}(u)}=\dist^2\paren{0,T_0(u)-u}=4\dist^2\paren{0,\lambda\paren{ P_A(u_1)- u_1}}.
	\end{align}
	On the other hand
	\begin{align}\label{eq:lefthandside}
		\dist^2\paren{\zetabar, \Psi_{g}(u)}
			&=\dist^2\paren{(0,0,0,0), \Psi_{0}(u)}\nonumber\\
			&=\dist^2\paren{(0,0,0,0), P_{\Omega_0}\Pi(u)-\Pi(u)}\nonumber\\
			&=2\dist^2\paren{0,P_A(u_1)-u_1},
	\end{align}
	since $\Omega_0=B\times A\times A\times B$. Combining \eqref{eq:lefthandside}
	and \eqref{eq:righthandside} yields \eqref{item: technical ass} in \cref{prop:metric subreg}
	with $\sigma=\frac{1}{\sqrt{2}\lambda}$.
\end{proof}

\subsection{Local Linear Convergence of \texorpdfstring{$\raar$}{T {DRlambda}}}\label{s:lin conv of raar}

\begin{lem}[uniqueness of difference vector for fixed points of $\raar$]\label{lem: uniqueness of zeta}
	Let $\lambda \in (0,1)$, and let $\xbar$ be a point in $\Fix \raar$ where 
    $A,B\subset \Euclid$ satisfy the assumptions 
	of \cref{thm: singlevalued at fixed points} with neighborhoods 
	$U(A,\epsilon,\xbar)$ and $U(B,\epsilon,\xbar)$. Then 
	$\{\zetabar\}= \mathcal{Z}(\xbar,g)\subset \Euclid^4$ for
	$\{g\} = P_B(\xbar)-P_A(P_B(\xbar))$ is unique and given by
	\begin{align*}
		\zetabar=(\zetabar_1,\dots,\zetabar_4)=\paren{g,-\tfrac{\lambda}{1-\lambda}g,-g,\tfrac{\lambda}{1-\lambda}g}.
	\end{align*}
\end{lem}
\begin{proof}
	By definition \eqref{eq:Z(u,g)} $\mathcal{Z}(x, g)$ is given by 
	\begin{align*}
		\mathcal{Z}(x, g)&\equiv\set{\zeta\equiv z-\Pi z }{ z \in W_0(g)\subset\Euclid^4, \ z_1=x},
	\end{align*}
	for 
	\begin{align*}
		W_0(g) \equiv \set{u \in \Euclid^4}{ u_1\in P_{B-\frac{\lambda}{1-\lambda}g}(u_2),%
 \ u_2 \in P_{A-\frac{\lambda}{1-\lambda}g}(u_3), \ u_3\in P_A(u_4), \ u_4\in P_B(u_1)}.
	\end{align*}
	Thus, the uniqueness of $\zetabar$ is a direct implication of the 
	uniqueness of $g$ as seen in \cref{r:Fix char}\eqref{r:Fix char iii}.
\end{proof}

Now we are ready to present the main result. The proof is based 
on the basic convergence result 
\cref{thm:2.18} and the facts from  \cref{s: charac of fixed points} 
and \cref{s: conv analysis}.
\begin{thm}[local linear convergence of $\raar$]\label{thm:conv raar}
	Let $\lambda \in (0,1)$, and let $\xbar$ be a point in $\Fix \raar$ where 
    $A,B\subset \Euclid$ satisfy the assumptions 
	of \cref{thm: singlevalued at fixed points} 
	with neighborhoods $U(A,\epsilon,\xbar)$ and $U(B,\epsilon,\xbar)$.  
	Set $\{g\} = P_B(\xbar)-P_A(P_B(\xbar))$ and 
    $\{\zetabar\}= \mathcal{Z}(\xbar,g)$ ($\zetabar=(\zetabar_1,\dots,\zetabar_4)\in \Euclid^4$). 
	Suppose that, at all $x\in \Fix \raar$ with $g\in P_B(x) - P_AP_B(x)$, 
	the sets $A,B\subset \Euclid$ satisfy the assumptions 
	of \cref{thm: singlevalued at fixed points} with corresponding 
	neighborhoods $U(A,\epsilon,x)$ and $U(B,\epsilon,x)$.
	Define the set 
	\begin{equation}\label{eq: S0}
		S_0\equiv \left\{x\in\Fix\raar~|~ \{g\} = P_B(x)-P_A(P_B(x))\right\}
	\end{equation}	 
	and let 
	\begin{equation}
		S_j \equiv \paren{S_0-\sum_{i=1}^{j-1}\zetabar_i}\quad \paren{j=1,2,3,4}.\label{eq:80}
	\end{equation}
	Fix some $\epsilon>0$ and define the neighborhood $U_A\equiv \cup_{x\in S_0}U(A, \epsilon, x)$ and likewise 
	$U_B\equiv \cup_{x\in S_0}U(B, \epsilon, x)$. Then 
	\[
		U\equiv \paren{U_B-\tfrac{\lambda}{1-\lambda}g}\times \paren{U_A-\tfrac{\lambda}{1-\lambda}g} \times U_A\times U_B
	\]
	is a neighborhood of $S\equiv S_1\times S_2 \times S_3\times S_4$. 		
	Suppose that, for $\Lambda\subseteq W(\zetabar)$ satisfying $S\subset\Lambda$ with  
	$\mmap{T_{\zetabar}}{\Lambda}{\Lambda}$, the following hold for all 
	$\ubar=\paren{\ubar_1,\ubar_2,\ubar_3,\ubar_4} \in S$:
	\begin{enumerate}[(i)]
		\item\label{ass:2} 
			for all $(\ubar_3,\ubar_4) \in S_3\times S_4$, the collection of sets 
			$\klam{A,B }$ is subtransversal 
			at $(\ubar_3,\ubar_4)$ for $(\ubar_3,\ubar_4)-\Pi (\ubar_3,\ubar_4)$ relative to 
			$\Lambda'\equiv \set{u=(u_1,u_2)\in \Euclid^2}{\paren{u_2-\frac{\lambda}{1-\lambda}g,u_1-\frac{\lambda}{1-\lambda}g,u_1,u_2}\in \Lambda}$ 
			with constant $\kappa$ on the neighborhood $U_A\times U_B$;
		\item\label{ass:3} for $\Phi_\zetabar\equiv T_\zetabar -\Id$ and 
			$\Psi_g\equiv P_{\Omega_g}\Pi-\Pi$	
			there exists a positive constant $\sigma$ such that 
			\begin{align}\label{eq: additional ineq}
				\dist\paren{\zetabar, \Psi_{g}(u)}\leq \sigma \dist\paren{0,\Phi_{\zetabar}(u)}
			\end{align}
			holds whenever $u \in \tilde{\Lambda}\cap U$ with $u_1\in B-\frac{\lambda}{1-\lambda}g$
			and 
			\[\tilde{\Lambda}\equiv \set{u\in \Lambda}{u=\paren{x_2-\frac{\lambda}{1-\lambda}g,x_1-\frac{\lambda}{1-\lambda}g,x_1,x_2} \text{for some }x_1,x_2\in \Euclid}.\]
	\end{enumerate}
	Then there exists an $\epsilon'\leq \epsilon$ and a neighborhood $U'\subset U$ 
	($U'=U'_1\times U'_2\times U'_3\times U'_4\subset \Euclid^4$) of $S$ on 
	which the sequence $\paren{u^k}_{k \in \NN}$ generated by  $u^{k+1}\in T_{\zetabar}u^k$ 
    seeded by a point $u^0 \in W\paren{\zetabar}\cap U'$ 
	with $u^0_1\in U_1'\cap \paren{B-\tfrac{\lambda}{1-\lambda}g}$ satisfies
	\begin{align*}
		\dist\paren{u^{k+1}, \Fix T_{\zetabar}\cap S}\leq c \dist \paren{u^k, S}\quad \paren{\forall k \in \NN}
	\end{align*}
        for
	\begin{align}\label{eq:82}
		c \equiv \sqrt{1+\epsilon'-\frac{1}{\bar{\kappa}^2}}<1
	\end{align}
	where $\bar{\kappa}=\kappa\sigma$ with $\kappa$ and $\sigma$ given by \eqref{ass:2} and \eqref{ass:3}.
	Consequently, 	$\dist\paren{u^k, \tilde{u}}\rightarrow 0$ for 
	some $\tilde{u}\in \Fix T_{\zetabar}\cap S$, and hence 
	\[\dist\paren{u^k_1, \tilde{u}_1}\rightarrow 0\] at least R-linearly with rate $c<1$.
	If $\Fix \raar \cap S_1$ is a singleton, then convergence is Q-linear.
\end{thm}
\begin{rem}[atlas for the assumptions]
	At first sight the assumptions in \cref{thm:conv raar} might seem overwhelming. 
	To provide some insight into the statement we discuss 
	the most important parts of the setting.
	\begin{enumerate}
		\item The assumptions of \cref{thm: singlevalued at fixed points} are needed to conclude 
		         almost averagedness of $\raar$.
		\item The requirement that the assumptions of \cref{thm: singlevalued at fixed points} hold at all 
			fixed points with the same gap vector is achieved by restricting our analysis
			to the set $S_0$. This also implies that we are considering only fixed points that 
			are {\em isolated relative to $\Lambda$}.
		\item Although we were not able to prove metric subregularity for a mapping related to
			$\raar$ directly, we can show this property for $T_\zetabar$ on $\Euclid^4$. In particular,
			assumptions \eqref{ass:2} and \eqref{ass:3} are used to guarantee metric subregularity
			from \cref{prop:metric subreg}.  Assumption \eqref{ass:2} guarantees subtransversality 
			of the collection
			$\klam{B-\frac{\lambda}{1-\lambda}g, A-\frac{\lambda}{1-\lambda}g, A , B}$ 
			since we have seen in 
			\cref{lem: addition perserves subtransversality} that subtransversality
			is preserved under the addition of some constant vector, here $\frac{\lambda}{1-\lambda}g$.
		\item The definitions of $\Lambda'$ and $\tilde{\Lambda}$ relate to the
			construction of the lifted product space version of the problem.
		\item The violation $\epsilon$ depends on the violations in 
			\cref{defn: super-reg+} as seen in \cref{thm: singlevalued at fixed points}. Thus,
			fixing some violation $\epsilon$ corresponds to certain choices of
			neighborhoods $U(A,\epsilon,\xbar)$ and $U(B, \epsilon, \xbar)$ and 
			violations $\epsilon_A$ and $\epsilon_B$ of \eqref{eq: epsilon subregularity}
			for the sets $A$ and $B$ respectively.
	\end{enumerate}
\end{rem}
\begin{proof}[Proof of \cref{thm:conv raar}]
	First, note that $U$ is a neighborhood of $S$ since $U_A\times U_B$ is a neighborhood of
	$S_3 \times S_4$, since for every $(u,\tilde{u})\in S_3\times S_4$ there exist $x,\tilde{x}\in S_0$
	such that $U(A, \epsilon, x)\times U(B, \epsilon, \tilde{x})\subset U_A\times U_B$ 
	is a neighborhood of $(u,\tilde{u})$.

	The neighborhood $U$ can be replaced by an enlargement of $S$, hence the result follows from \cref{thm:2.18}
	once it can be shown that the assumptions are satisfied for the mapping
	$T_{\zetabar}$ on the product space $\Euclid^4$ restricted to $\tilde{\Lambda}$. 
	
	To do so, we note that $\raar$ is almost firmly nonexpansive at each $\tilde{y}\in S_1$
	on $U_B$ by \cref{thm: singlevalued at fixed points} since the assumptions
	\eqref{ass B}-\eqref{ass 2} of \cref{thm: singlevalued at fixed points} are satisfied.
	Thus, $\alpha=1/2$.
	Likewise the violation is given by
	$\epsilon$ on $U_B$.
	Since $T_{\zetabar}$ is just $\raar$ shifted by $\zetabar$ 
	on the product space, it follows that $T_{\zetabar}$ is pointwise almost averaged 
	at $y \in S\equiv S_1\times S_2\times S_3\times S_4$ with the same violation 
	$\epsilon$ and averaging constant $\alpha=1/2$ on $U$.
	
	By \cref{lem: addition perserves subtransversality} and \cref{rem:capser} therefore, 
	assumption \eqref{ass:2} implies that 
	for $\ubar =(\ubar_1, \ubar_2, \ubar_3,\ubar_4) \in S$, the collection of sets 
	\[\klam{B-\frac{\lambda}{1-\lambda}g,A-\frac{\lambda}{1-\lambda}g,A,B }\] is subtransversal 
	at $\ubar$ for $\bar{\zeta }\equiv \ubar-\Pi \ubar$ relative to $\tilde{\Lambda}$ with constant 
	$\kappa$ on the neighborhood $U$, hence \cref{thm:2.18}\eqref{ass:a} is satisfied. 
	Moreover, assumption
	\eqref{ass:3} of \cref{thm:conv raar} and \cref{prop:metric subreg} with 
	$U\cap\tilde{\Lambda}\subset\mathcal{N}\equiv \set{z \in \Euclid^4}{P_A(z_4+\tfrac{\lambda}{1-\lambda}g)=z_3}$ 
        by \cref{thm: singlevalued at fixed points}\eqref{ass 2} yield assumption \cref{thm:2.18}\eqref{ass:b}.
        In total, the assumptions of \cref{thm:2.18} are all satisfied
	for $T_\zetabar$ on $\Euclid^4$ restricted to $\tilde{\Lambda}$, 
	and thus we conclude that \eqref{e:iterates} holds.  

	What remains is to show 
        that \eqref{eq:kappa-epsilon} holds, which would imply at least R-linear convergence.
	To achieve this choose some $\epsilon'>0$ with $\epsilon'<\epsilon$
	such that \eqref{eq:kappa-epsilon} is satisfied.
	By \cref{th:JEFF the Brotherhood} we can always find neighborhoods $U(B,\epsilon',x)\subset U(B,\epsilon,x)$ and
	$U(A,\epsilon',x)\subset U(A,\epsilon,x)$ for all $x\in S_0$ that satisfy the assumptions of
	\cref{thm: singlevalued at fixed points}. Following the constructions above
	we define $U_A'\equiv \cup_{x\in S_0}U(A, \epsilon', x)$ and 
	$U_B'\equiv \cup_{x\in S_0}U(B, \epsilon', x)$ and get $U_A'\subset U_A$ as well
	as $U_B'\subset U_B$. Thus, all the properties that we have shown to be true on $U$
	also hold on the subset $U'$ defined by
	\[
		U'\equiv \paren{U_B'-\tfrac{\lambda}{1-\lambda}g}\times \paren{U_A'-\tfrac{\lambda}{1-\lambda}g} \times U_A'\times U_B'.
	\]
	In particular, the constants $\kappa$ and $\sigma$ in \eqref{ass:2} and \eqref{ass:3} also suffice for the smaller 
	neighborhoods $U'_A\times U'_B$ and $U'$.  
	As a consequence, the assumptions of \cref{thm:2.18} are all satisfied 
	and \eqref{eq:kappa-epsilon} holds which implies at least R-linear 
	convergence to $\tilde{u}$. Since $\tilde{u}_1\in \Fix \raar \cap S_1$, this completes the proof.
\end{proof}

\begin{rem}[a closer look at the convergence statement]
The gap vector $g$ and difference vector $\zetabar$ in \cref{thm:conv raar} rely on the
structure of the intersection of the sets $A$ and $B$. The consistent case, that is $A\cap B\neq \emptyset$,
leads to a simplification of the problem. Here, the gap is $0$. Similarly the related difference vector
is of the form $\zetabar=\klam{0,0,0,0}$. Hence, the assumptions which involve at least one of these vectors 
can be simplified.
When the intersection $A\cap B$ is empty, namely the inconsistent case, the value of $\zetabar$ is dependent
on the choice of $\lambda$. We distinguish three important cases.
\begin{enumerate}
	\item $\lambda=\frac{1}{2}$. Here $\frac{\lambda}{1-\lambda}$ reduces to 1. As a result the phantom sets
		are shifted by the entire gap $g$ such that $A$ and $B-g$ have a common point. The
		difference vector is of the form $\zetabar=\klam{g, -g,-g,g}$.
	\item $\lambda\rightarrow 1$. Then $\frac{\lambda}{1-\lambda}\to +\infty$. 
		That is, the phantom sets recede to the horizon in the direction $-g$.
 	\item $\lambda \rightarrow 0$. In this case $\frac{\lambda}{1-\lambda}$ converges to $0$ and 
		the phantom sets coincide in the limit with the original ones. So, $\Omega_{u,g}=B\times A\times A \times B$. 
		Cyclic projections on these sets $\klam{B,A,A,B}$ in the given order is nothing more 
		than alternating projections between the sets $A$ and $B$.  At $\lambda=0$, however, 
		$\Fix \raar=B$, which is clearly larger than the fixed point set for alternating projections. 
\end{enumerate}
\end{rem}
Although the individual assumptions can be challenging to prove, as we will see in \cref{sec:examples},
they can reduce to a simpler form if we consider a convex and consistent setting.
The reason for this is twofold. First, subtransversality at points in the intersection is
nothing more than {\em local linear regularity} of the collection of sets, \cite[Proposition 3.3]{LukNguTam17}. 
Second, it was shown
that local linear regularity is equivalent to the global property of linear regularity
in the setting of closed convex sets, as seen in \cite[Theorem 6.1]{BakDeuLi2005}. Thus, assuming
\eqref{eq:ass22} locally for closed and convex sets implies that this property holds globally.
To prove this statement we will first present the auxiliary statements, which are essential 
to show the global convergence result.
\begin{propn}[subtransversality at common points]\cite[Proposition 3.3]{LukNguTam17}\label{prop:lin reg}
	Let $\Euclid^m$ be endowed with the 2-norm, that is, 
	$\norm{\paren{x_1,\dots, x_m}}_2=\paren{\sum_{j=1}^{m}\norm{x_j}_{\Euclid}^2}^{1/2}$. 
	A collection $\klam{\Omega_1, \Omega_2, \dots, \Omega_m}$ of nonempty and closed 
	subsets of $\Euclid$ is	subtransversal relative to 
	\[\Lambda\equiv \set{x=\paren{u,u, \dots, u}\in \Euclid^m}{| u \in \Euclid}\]
	at $\xbar=\paren{\ubar, \ubar, \dots, \ubar}$ with $\ubar \in \cap_{j=1}^m \Omega_j$ for
	$\ybar=0$ with constant $\kappa$ if there exist a neighborhood $U'$ of $\ubar$ together with a 
	constant $\kappa'$ satisfying $\sqrt{m}\kappa'\leq \kappa$ such that
	\begin{align}\label{eq:67}
		\dist\paren{u, \cap_{j=1}^m\Omega_j}\leq \kappa'\max_{j=1, \dots, m}\dist\paren{u, \Omega_i}, \quad \forall \ u \in U'.
	\end{align}
	Conversely, if $\klam{\Omega_1, \Omega_2, \dots, \Omega_m}$ is subtransversal relative 
	to $\Lambda$ at $\xbar$ for $\ybar=0$ with constant $\kappa$, then \eqref{eq:67} is satisfied 
	with any constant $\kappa'$ for which $\kappa\leq \kappa'$.
\end{propn}
The property in \eqref{eq:67} is called {\em local linear regularity}. 
If the inequality holds for all $u \in \Euclid$, the collection $\klam{A, B}$
is said to be {\em linearly regular}.
Bakan, Deutsch and Li showed in \cite{BakDeuLi2005} the equivalence of both properties
when the sets are closed and convex.
\begin{lem}\label{lem:cvx local lin reg}\cite[Theorem 6.1]{BakDeuLi2005}
	Let the sets $A$ and $B$ be nonempty closed convex sets with nonempty 
	intersection, i.e. $A\cap B\neq \emptyset$. Then the following are equivalent.
	\begin{enumerate}[(i)]
		\item There is a $\delta>0$ such that the collection of sets is locally linearly
			regular at $\xbar\in A\cap B$ on $\Ball_\delta(\xbar)$.
		\item The collection of sets is linearly regular at $\xbar\in A\cap B$.
	\end{enumerate}
\end{lem}
Having \cref{prop:lin reg} and \cref{lem:cvx local lin reg} we are now
ready to state a global convergence result for closed convex sets. 
\begin{cor}[global convergence in the consistent and convex setting]\label{cor: local cvgz in cvx setting}
	Let $\lambda \in (0,1)$, and let $\xbar$ be a point in $\Fix \raar$. Moreover, let both $A$ and $B$
	be closed and convex with nonempty intersection, i.e. $A\cap B\neq \emptyset$
	and therefore $\Fix \raar=A\cap B$. 
	Then $\klam{g}=P_B(\xbar)-P_A(P_B(\xbar))=0$ and $\klam{\zetabar}=\Zcal(\xbar,g)=\klam{0}$
	($\zetabar=\paren{\zetabar_1,\zetabar_2,\zetabar_3,\zetabar_4}\in \Euclid^4$).
	Define the set
	\begin{align*}
		S_0\equiv \Fix\raar=A\cap B.
	\end{align*}
	Suppose that, the following hold for all $\ubar=(\ubar_1,\ubar_2) \in S\equiv S_0\times S_0$:
	\begin{enumerate}[(i)]
		\item\label{eq:ass22} the collection of sets 
			$\klam{A,B }$ is subtransversal 
			at $(\ubar_1,\ubar_2)$ for $(\ubar_1,\ubar_2)-\Pi (\ubar_1,\ubar_2)$ relative to 
			$\Lambda'\subset\set{u\in \Euclid^2}{u_1=u_2}$
			with constant $\kappa$ on some neighborhood $U' \subset \Euclid^2$ ($U'= U_A\times U_B$).
	\end{enumerate}
	Then the sequence $\paren{x^k}_{k\in \NN}$ generated by $x^{k+1}\in \raar(x^k)$ seeded by a point 
	$x^0\in \Lambda'\cap U_B$
	satisfies
	\begin{align*}
		\dist\paren{x^{k+1}, \Fix \raar}\leq c \dist \paren{x^k, \Fix \raar}\quad \paren{\forall k \in \NN}
	\end{align*}
	for  
	\begin{align*}
		c\equiv \sqrt{1-\frac{2\lambda^2}{\kappa^2}}<1
	\end{align*}
	with $\kappa$ by \eqref{eq:ass22}.
	Consequently, 	$\dist\paren{x^k, \tilde{x}}\rightarrow 0$ for 
	some $\tilde{x}\in \Fix \raar$ at least R-linearly with rate $c<1$. 
	If $\Fix \raar$ is a singleton, then convergence is Q-linear.
\end{cor}
\begin{rem}[global convergence for convex sets]
	There are only two changes from \cref{thm:conv raar} to \cref{cor: local cvgz in cvx setting}.
	First, the sets are required to be convex. Thus, convergence in general is guaranteed 
	by Opial \cite{Opial1967} as noted in the beginning of this section. 
	Moreover, the local assumption \eqref{eq:ass22} in this case is a global one, i.e. 
	$U'=\Euclid^2$, by  \cref{prop:lin reg} and \cref{lem:cvx local lin reg}.
	The second difference, assumption  \eqref{ass:3} in \cref{thm:conv raar}, is
	always satisfied by \cref{propn:technical ass for consistent} since $\Fix \raar =A\cap B$.
\end{rem}
\begin{proof}[Proof of \cref{cor: local cvgz in cvx setting}]
	Since both $A$ and $B$ are convex, not only the difference vector is unique as seen in \cref{lem: uniqueness of zeta},
	but also the gap vector $g$ for any fixed point in $\Fix \raar$. Thus, $S_0=\Fix \raar$. 
	$\Fix \raar=A\cap B$ by \cref{thm:fixed points}. With these observations we get 
	immediately that the sets involved in \cref{thm:conv raar} simplify to the 
	following 
	\begin{align*}
		S &=S_0\times S_0 \times S_0 \times S_0,\\
		W(\zetabar) &= \set{u\in \Euclid^4}{u-\Pi u=0}=\set{u\in \Euclid^4}{u_1=u_2=u_3=u_4},\\
		U &= U_B \times U_A \times U_A\times U_B,\\
		\Lambda'&\subset\set{u\in \Euclid^2}{u_1=u_2},
	\end{align*}
	since $\Lambda\subset W(\zetabar)$. Thus, assuming \eqref{eq:ass22} in 
	\cref{cor: local cvgz in cvx setting} is equivalent to assuming \cref{thm:conv raar}\eqref{ass:2} 
	in the convex and consistent setting.	
	Moreover,
	since the sets $A$ and $ B$ are convex, both the projector and the reflector onto these sets are single-valued 
	(see for example \cite[Theorem 3.14]{BauCom11}). Additionally the projection is 
	firmly nonexpansive, \cite[Proposition 4.8]{BauCom11}, and thus the reflector nonexpansive,
	\cite[Proposition 4.2]{BauCom11}, which implies that $\raar$ is averaged with constant
	$\alpha=(1/2)$. The conditions of \cref{thm: singlevalued at fixed points}
	are therefore satisfied with neighborhoods chosen to be $\Euclid$. Also, since the sets $A$ and $B$
	are convex, they are super-regular at a distance by \cref{prop: cxv set is super-reg at dist}
	with $\epsilon=0$. 
	Since every fixed point is an element of the intersection $A\cap B$, we deduce
	by \cref{propn:technical ass for consistent} that assumption \eqref{ass:3} of 
	\cref{thm:conv raar} holds.	
	The local convergence result follows then from \cref{thm:conv raar}. What is left to show
	is the global convergence property. 
	
	By \eqref{eq:ass22} and \cref{prop:lin reg} the collection of sets $\klam{A,B}$ is locally
	linearly regular on $U'$. Thus, there exists a $\delta>0$ such that $\klam{A,B}$ is locally
	linearly regular on $\Ball_\delta(\xbar)$. Using \cref{lem:cvx local lin reg}
	we get that $\klam{A,B}$ is linearly regular since $A$ and $B$ are convex sets. In total,
	\eqref{eq:ass22} holds with $U'=\Euclid^2$. That is, the assumption holds globally. 
	Since \eqref{ass:3} of \cref{thm:conv raar} holds globally as well by 
	\cref{propn:technical ass for consistent}, the assumptions
	of the underlying convergence framework in \cref{thm:2.18} hold globally. Therefore, the 
	sequence converges globally, which completes the proof.
\end{proof}
\begin{rem}[linking our results to already existing literature] As noted in the introduction, 
the works \cite{HesseLuk2013, Phan16, LiPong16, LukNguTam17,min18} all 
analyze the Douglas-Rachford algorithm for consistent nonconvex feasibility.  In 
\cite{LukNguTam17} the framework used here was applied to Douglas-Rachford for structured 
nonconvex optimization. In \cite{min18} the authors showed local R-linear convergence for superregular sets
intersecting linear regular.
Our analysis of relaxed Douglas-Rachford includes or subsumes that of 
all previous studies in the context of set feasibility, with the exception of \cite{LiPong16}, 
which addresses global convergence guarantees for consistent 
feasibility.  The assumptions of that paper, namely compactness and semi-algebraicity (not 
to mention nonempty intersection) are different than the notions that we work with. 
Certainly compactness is a regularity assumption, as is semi-algebraicity or its more general 
Kurdyka-\L ojasiewicz-type regularity, 
but these notions serve a different purpose.  Indeed, even convex sets need not be 
semi-algebraic or compact.  This suggests that Kurdyka-\L ojasiewicz-type regularity 
and compactness could be 
properties {\em in addition} to the ones we use in order to arrive at global statements.  
Nevertheless, as shown in \cref{cor: local cvgz in cvx setting}, in the convex case, the local 
analysis suffices to infer global convergence properties.  
A more thorough study of the relationship between the different notions of regularity 
would be fruitful, but is beyond the scope of our paper. 

Our results could be extended to sets with even weaker regularity, namely $\epsilon$-subregular sets 
instead of super-regular sets at a distance under the additional assumption that suitable neighborhoods exist.
But, the present setting is technical enough - increased generality would have only made the details even more difficult to parse.  
Moreover,  the advantage of this specific type of nonconvexity is that we are not only able
	to present existence results on neighborhoods where we get local convergence,  but 
	we are able to construct the neighborhoods explicitly.
\end{rem}


\section{Elementary Examples}\label{sec:examples}
We demonstrate in this section explicit verification of the assumptions of Theorem \cref{thm:conv raar} for 
a typical class of problems.  In particular, we consider the configurations that arise with feasibility 
problems involving intersecting and nonintersecting {\em spheres}.  
This is of particular interest for the {\em source localization problem}
and the {\em phase retrieval problem}, especially the nonintersecting case.  The idealized source localization problem amounts to finding the 
intersection of spheres that are determined by distance measurements to receivers whose locations are 
known.  When the distance measurements are noisy, or the given locations of the receivers are inaccurate, 
the intersection over all spheres will be empty almost always.  For phase retrieval, the measurements are pointwise amplitude measurements
in the Fourier domain of an unknown object.  In other words, the constraint sets are two-dimensional spheres 
in the image of a linear transformation.  Since both the measurements and the object are assumed to 
have compact support, the phase retrieval problem in diffraction imaging is fundamentally inconsistent.
In our development below, we carry out the explicit calculations to 
verify the assumptions of Theorem \cref{thm:conv raar} for circles (spheres in $\mathbb{R}^2$) 
which was shown in \cite{LukSabTeb18} to be the fundamental geometry for phase retrieval and source location problems.  
Affine subspaces are 
included as spheres centered at infinity.  

There are 5 distinct cases to consider:
1. intersecting circles, 2. nonintersecting separable circles, 3. nonintersecting, nonseparable, nonconcentric circles, 4. nonintersecting concentric circles, and 5. tangential circles.  We show that the verification
can be carried out ``by hand'' in the first example.  For the sake of brevity, the verification is carried out in 
the remaining examples with the help of symbolic computation.  We were unable to prove or disprove
that the required conditions hold in Example 4.  In Example 5 we determine that the assumptions are 
{\em not} satisfied, and therefore the algorithm cannot converge linearly.  The symbolic worksheets where our calculations were
carried out are available at \url{http://vaopt.math.uni-goettingen.de/en/publications.php}.

For this entire section let $R$ be a positive real-valued number and $\lambda\in (0,1)$ if not specified. 
To verify the subtransversality and the technical condition \eqref{ass:3} in \cref{thm:conv raar}
we often did not calculate the constants explicitly but bounded them from below. That is,
\begin{align*}
	\kappa >&\frac{\dist\paren{u,\Psi_g^{-1}(\zetabar)\cap W(\zetabar)}}{\dist\paren{\zetabar,\Psi_g(u)}},\\
	\sigma >&\frac{\dist\paren{\zetabar, \Psi_{g}(u)}}{\dist\paren{0,\Phi_{\zetabar}(u)}},
\end{align*}
where $\kappa$ was the constants of subtransversality and $\sigma$ describes the technical condition.
In this subsection we will deal with neighborhoods of fixed points. As a consequence
the constants computed bound the rate of linear convergence from below in such cases.
Note that we can always find a neighborhood such that the convergence is linear for examples
consisting of two circles by \cref{thm: singlevalued at fixed points} and \cref{eg:1}.
\begin{eg}[two intersecting circles]\label{eg: 1. case}	
	The first example consists of two circles intersecting at exactly two points.  
	Without loss of generality we can restrict the analysis to the following setting
	\begin{align*}
		A\equiv &\set{x\in \Rbb^2}{ \norm{x}=1}\\
		B\equiv &\set{x\in \Rbb^2 }{ \norm{x-(0,a)}=R}, 
	\end{align*}
	where $a\in \Rbb \setminus \klam{0}$ and 
	$R \in \paren{\min_{y \in A}\dist\paren{\paren{0,a},y},\max_{y \in A}\dist\paren{\paren{0,a},y}}$.
	\begin{figure}[ht]
		\centering
		\begin{tikzpicture}
			\draw[thick,->] (-2,0) -- (2,0);
			\draw[thick,->] (0,-3) -- (0,1.5);
			\draw (0,0) circle (1cm);
			\draw (0.7,0.3) node {$A$};
			\draw (0,-1.5) circle (1cm);
			\draw (0.7,-1.2) node {$B$};
		\end{tikzpicture}
		\caption{\cref{eg: 1. case} for $a=-1.5$ and $R=1$}
	\end{figure}
	Note that the endpoints of the interval for $R$ correspond to the setting of two touching circles, 
	see \cref{eg: 5. case}.	

	First we consider the points in the intersection $A\cap B$, namely
	\begin{align*}
		\paren{\pm\sqrt{1-\paren{\frac{1-R^2+a^2}{2a}}^2},\frac{1-R^2+a^2}{2a}}.
	\end{align*}
	Due to the symmetry of the problem we restrict the analysis to the point 
	$\paren{+\sqrt{1-\paren{\frac{1-R^2+a^2}{2a}}^2},\frac{1-R^2+a^2}{2a}}$.
		
	The following statements regarding the assumptions made in \cref{thm:conv raar} are easily
	verified either by hand or with the help of symbolic computation.
	\begin{enumerate}[(i)]
		\item $S_0\equiv \klam{\paren{+\sqrt{1-\paren{\frac{1-R^2+a^2}{2a}}^2},\frac{1-R^2+a^2}{2a}}}\in \Fix \raar$
		\item In $\Rbb_+\times \Rbb$ there is a unique fixed point.
			$\xbar=\paren{\ubar,\ubar,\ubar,\ubar}$, where $\ubar =\paren{+\sqrt{1-\paren{\frac{1-R^2+a^2}{2a}}^2},\frac{1-R^2+a^2}{2a}}$.
		\item The difference vector is unique and given by $\zetabar=\paren{(0,0),(0,0),(0,0),(0,0)}$, since
			$\ubar\in A \cap B$.
		\item The sets $A$ and $B$ satisfy the assumptions of \cref{thm: singlevalued at fixed points}
			at $\ubar$ with neighborhoods $U_1$ and $U_2$ being open balls around $\ubar$, that is
			$\Ball_\delta(\ubar)$, for $\delta\in (0,1)$. This can be shown similar to \cref{eg:1}.
		\item This example considers a setting with nonempty intersection. By \cref{prop:lin reg} one can equivalently prove linear regularity to get subtransversality in such instances.
            Our aim is to use this statement to prove that \cref{eg: 1. case} satisfies \eqref{eq:67}.
			
			For this we select any $u=(u_1, u_2) \in U_1\cap A$ where $U_1=U_2$ and $u_1>0$. 
			Such a point exists since the statements in \cref{thm:conv raar} are all with respect 
			to the set $\Lambda$ which is a subset of $W(\bar{\zeta})$. Thus the restriction to one of the sets 
			is no contradiction. The condition $u_1>0$ ensures that we always project on the chosen point in the intersection: 
			$\paren{+\sqrt{1-\paren{\frac{1-R^2+a^2}{2a}}^2},\frac{1-R^2+a^2}{2a}}$.  Then the condition 
			\begin{align*}
				\dist\paren{u, A\cap B}\leq \kappa'\max\klam{\dist\paren{u, A},\dist\paren{u, B}},
			\end{align*}
			simplifies to
			\begin{align*}
				\dist\paren{u, A\cap B}\leq \kappa'\dist\paren{u, B},
			\end{align*}
			which we reformulate in the following to 
			\begin{align}\label{eq:linreg}
				\norm{u-P_{A\cap B}(u)}\leq \kappa'\norm{u-P_{B}(u)}.
			\end{align}
			Note that \eqref{eq:linreg} implies \eqref{eq:67} since $u\in A$ and thus 
			implies linear regularity.
	
			Next, we show \eqref{eq:linreg}.
			\begin{align*}
				\norm{u-P_{A\cap B}(u)}&\leq \norm{u-P_{B}(u)}+\norm{P_{B}(u)-P_{A\cap B}(u)}\\
				&\leq \norm{u-P_{B}(u)}\paren{1+\frac{\norm{P_{B}(u)-P_{A\cap B}(u)}}{\norm{u-P_{B}(u)}}}.
			\end{align*}
			What remains to show is that 
			$1+\frac{\norm{P_{B}(u)-P_{A\cap B}(u)}}{\norm{u-P_{B}(u)}}$ is bounded above
			by a nonnegative constant.
	
			By construction we get for the individual projections
			\begin{align*}
				P_{A\cap B}(u)&=\paren{+\sqrt{1-\paren{\frac{1-R^2+a^2}{2a}}^2},\frac{1-R^2+a^2}{2a}},\\
				P_{B}(u)&=\paren{0,a}+\frac{u-\paren{0,a}}{\norm{u-\paren{0,a}}}R.
			\end{align*}
			Inserting this in the above expression yields
			\begin{align*}
				1+\frac{\norm{P_{B}(u)-P_{A\cap B}(u)}}{\norm{u-P_{B}(u)}}&=1+\frac{\norm{\paren{0,a}+\frac{u-\paren{0,a}}{\norm{u-\paren{0,a}}}R-\paren{\sqrt{1-\paren{\frac{1-R^2+a^2}{2a}}^2},\frac{1-R^2+a^2}{2a}}}}{\norm{u-\paren{0,a}+\frac{u-\paren{0,a}}{\norm{u-\paren{0,a}}}R}}\\
				&\leq 2+\frac{\norm{u-\paren{\sqrt{1-\paren{\frac{1-R^2+a^2}{2a}}^2},\frac{1-R^2+a^2}{2a}}}}{\norm{u-\paren{0,a}+\frac{u-\paren{0,a}}{\norm{u-\paren{0,a}}}R}}\\
				&< 2+\frac{1}{\norm{u-\paren{0,a}+\frac{u-\paren{0,a}}{\norm{u-\paren{0,a}}}R}},
			\end{align*}
			since $u \in A$ and thus $\norm{u-\paren{\sqrt{1-\paren{\frac{1-R^2+a^2}{2a}}^2},\frac{1-R^2+a^2}{2a}}}<1$. 
			The denominator 
			$${\norm{u-\paren{0,a}+\frac{u-\paren{0,a}}{\norm{u-\paren{0,a}}}R}}$$ 
			can be bounded as follows.   Since $u \in A$ we have 
			\begin{align*}
				0<\min_{y \in A}\dist\paren{\paren{0,a},y}\leq \norm{u-\paren{0,a}}\leq\max_{y \in A}\dist\paren{\paren{0,a},y}.
			\end{align*}
			Equivalently
			\begin{align*}
				\frac{1}{\min_{y \in A}\dist\paren{\paren{0,a},y}}\geq \frac{1}{\norm{u-\paren{0,a}}}\geq\frac{1}{\max_{y \in A}\dist\paren{\paren{0,a},y}}.
			\end{align*}
			Thus,
			\begin{align*}
				1-\frac{R}{\norm{u-\paren{0,a}}}&\geq 1-\frac{R}{\max_{y \in A}\dist\paren{\paren{0,a},y}}, \quad \text{and}\\
				\norm{u-\paren{0,a}}\paren{1-\frac{R}{\norm{u-\paren{0,a}}}}&\geq \min_{y \in A}\dist\paren{\paren{0,a},y}\paren{1-\frac{R}{\max_{y \in A}\dist\paren{\paren{0,a},y}}}\\
				\Rightarrow \frac{1}{\norm{u-\paren{0,a}}\paren{1-\frac{R}{\norm{u-\paren{0,a}}}}}&\leq \frac{1}{\min_{y \in A}\dist\paren{\paren{0,a},y}\paren{1-\frac{R}{\max_{y \in A}\dist\paren{\paren{0,a},y}}}}=:\kappa'.
			\end{align*}
			Since $R \in \paren{\min_{y \in A}\dist\paren{\paren{0,a},y},\max_{y \in A}\dist\paren{\paren{0,a},y}}$, 
			$\kappa'$ is bounded above.
			
			In total, $A\cap B$ is locally linear regular at 
			$\paren{+\sqrt{1-\paren{\frac{1-R^2+a^2}{2a}}^2},\frac{1-R^2+a^2}{2a}}$. By \cref{prop:lin reg}
			we deduce subtransversality with constant $\kappa\equiv \kappa'\sqrt{2}$.
		\item The technical condition \eqref{ass:3} in \cref{thm:conv raar} is satisfied with
			\begin{align*}
				\sigma^2=\frac{1}{2   \lambda^{2}}
			\end{align*}
			by \cref{propn:technical ass for consistent}.
	\end{enumerate}			
	Thus, the assumptions of \cref{thm:conv raar} are satisfied and the relaxed Douglas-Rachford
	algorithm converges locally linearly to $\ubar$ with rate
	$1>c>\sqrt{1-\tfrac{\lambda^2}{{\kappa'}^2}}$ as long as the starting point is close enough to $\ubar$.
	
	Similarly, this argument can be repeated for 
	$\paren{-\sqrt{1-\paren{\frac{1-R^2+a^2}{2a}}^2},\frac{1-R^2+a^2}{2a}}$,
	which shows that, in this situation, both subtransversality and the technical condition 
	at the two points in the intersection are satisfied.
	
	Note that the point $\paren{0,-1}$ does not to lead to a fixed point of the relaxed 
	Douglas-Rachford algorithm. Whereas for the Alternating 
	Projection method, defined by the operator $P_{A}P_{B}$, $\paren{0,-1}$ 
	is always a fixed point.	
	In particular, for any $\lambda\in (0,1)$ the fixed point set of $\raar$ does
	not contain any point of the form $(0,y)$ for $y \in \Rbb$.
\end{eg}
\begin{eg}[nonintersecting separable circles]\label{eg: 2. case}	
	This example consists of two circles in $\Rbb^2$ that are shifted by some vector 
	in $\Rbb^2$ such that they do not intersect in any point. Let $R>0$ and define
	\begin{align*}
		A&\equiv \set{x\in \Rbb^2 }{ \norm{x}=1}\\
		\text{and}\quad	B&\equiv \set{x\in \Rbb^2 }{ \norm{x-(2+R,0)}=R}.
	\end{align*}
		\begin{figure}[ht]
		\centering
		\begin{tikzpicture}
			\draw[thick,->] (-2,0) -- (5,0);
			\draw[thick,->] (0,-1.5) -- (0,1.5);
			\draw (0,0) circle (1cm);
			\draw (0.7,0.3) node {$A$};
			\draw (3,0) circle (1cm);
			\draw (3.7,0.3) node {$B$};
		\end{tikzpicture}
		\caption{\cref{eg: 2. case} for $R=1$}
	\end{figure}
	The only fixed point of $\raar$ on $A$ and $B$ is given by
	\begin{align*}
		\ubar=\paren{2,0}-\frac{\lambda}{1-\lambda}\paren{1,0}
	\end{align*}
	for $\lambda \in (0,1)$. The following statements regarding the assumptions made in
	\cref{thm:conv raar} are easily verified either by hand or with the help of symbolic
	computation.
	\begin{enumerate}[(i)]
		\item $S_0\equiv \Fix \raar =\klam{\ubar}$.
		\item The difference vector is unique as well and given by $\zetabar=\paren{(1,0),-\frac{\lambda}{1-\lambda}\paren{1,0},(-1,0),\frac{\lambda}{1-\lambda}\paren{1,0}}$.
		\item As noted on \cref{eg:1}\eqref{eg:1ii} the assumptions of 
			\cref{thm: singlevalued at fixed points} are satisfied 
			for neighborhoods chosen as tubes.
		\item We bounded the modulus of subtransversality $\kappa$ from below with the help of symbolic computation as follows 
			\begin{align*}
				\kappa^2>\frac{8   {\left(R^{2} + 2   R + 1\right)}}{R^{2} + 2   R + 5}.
			\end{align*}
		\item Likewise, we can bound the technical assumption \eqref{ass:3} in \cref{thm:conv raar} from below. Due to the complexity of the constant we omit it here and refer the reader to the Sage worksheet posted at \url{http://vaopt.math.uni-goettingen.de/en/publications.php}.
	\end{enumerate}
	Thus, the assumptions of \cref{thm:conv raar} are satisfied and the relaxed Douglas-Rachford
	algorithm converges locally linearly to $\ubar$ with rate $1>c>\sqrt{1-\tfrac{1}{(\kappa\sigma)^2}}$
	as long as the starting point is close enough to $\ubar$.
 \end{eg}
\begin{eg}[nonintersecting, nonseparable and nonconcentric circles]\label{eg: 3. case}
	This example consists of two sets having not the same center
	and one of the circles surrounds the other one. Let $R>0$ and set 
	\begin{align*}
		A&\equiv \set{x\in \Rbb^2 }{ \norm{x}=1}\\
		B&\equiv \set{x\in \Rbb^2 }{\norm{x-(0,-\frac{1}{2}-R)}=2+R}.
	\end{align*}
		\begin{figure}[ht]
		\centering
		\begin{tikzpicture}[scale=0.7]
			\draw[thick,->] (-3.5,0) -- (3.5,0);
			\draw[thick,->] (0,-5) -- (0,2);
			\draw (0,0) circle (1cm);
			\draw (0.7,0.3) node {$A$};
			\draw (0,-1.5) circle (3cm);
			\draw (1.4,0.6) node {$B$};
		\end{tikzpicture}
		\caption{\cref{eg: 3. case} for $R=1$}
	\end{figure}
	Our analysis considers the fixed point 
	\begin{align*}
		\ubar=\paren{0,\frac{3}{2}}-\frac{\lambda}{1-\lambda}\paren{0,\frac{1}{2}}
	\end{align*}
	of $\raar$ on $A$ and $B$ for $\lambda<2/3$. The following 
	statements regarding the assumptions made in
	\cref{thm:conv raar} are easily verified either by hand or with the help of symbolic
	computation.
	\begin{enumerate}[(i)]
		\item $S_0\equiv \klam{\ubar}\in\Fix \raar $.
		\item The difference vector is unique as well and given by 
			\[\zetabar=\paren{\paren{0,\frac{1}{2}}, 
		-\frac{\lambda}{1-\lambda}\paren{0,\frac{1}{2}},
		-\paren{0,\frac{1}{2}},
		-\frac{\lambda}{1-\lambda}\paren{0,\frac{1}{2}}}.\]
		\item Similar to the analysis made in \cref{eg:1}\eqref{eg:1ii} the assumptions of 
			\cref{thm: singlevalued at fixed points} are satisfied 
			for neighborhoods chosen as tubes.
		\item We bounded the modulus of subtransversality $\kappa$ from below with the help of symbolic computation as follows 
			\begin{align*}
				\kappa^2>\frac{9   {\left(4   R^{2} + 12   R + 9\right)}}{2   R^{2} + 6   R + 9}.
			\end{align*}
		\item Likewise, we can bound the technical assumption \eqref{ass:3} in \cref{thm:conv raar} from below. Due to the complexity of the constant we omit it here and refer the reader to the Sage worksheet posted at \url{http://vaopt.math.uni-goettingen.de/en/publications.php}.
	\end{enumerate}
	Thus, the assumptions of \cref{thm:conv raar} are satisfied and the relaxed Douglas-Rachford
	algorithm converges locally linearly to $\ubar$ with rate $1>c>\sqrt{1-\tfrac{1}{(\kappa\sigma)^2}}$
	as long as  as long as the starting point is close enough to $\ubar$.
\end{eg}
\begin{eg}[nonintersecting, nonseparable concentric circles]\label{eg: 4. case}	
	In comparison to \cref{eg: 3. case}, the only thing we change is that we do not allow
	the circles to have different centers anymore. Let $R>0$ and define
	\begin{align*}
		A\equiv \set{x\in \Rbb^2 }{ \norm{x}=1}, \qquad 
		B\equiv \set{x\in \Rbb^2 }{\norm{x}=R},
	\end{align*}
	where we restrict $R$ to be strictly greater that $1$, i.e. $R>1$\footnote{For $R<1$ we can 
	change the roles of $A$ and $B$, which results in the situation presented here.}.
		\begin{figure}[ht]
		\centering
		\begin{tikzpicture}
			\draw[thick,->] (-2.5,0) -- (2.5,0);
			\draw[thick,->] (0,-2.5) -- (0,2.5);
			\draw (0,0) circle (1cm);
			\draw (0.7,0.3) node {$A$};
			\draw (0,0) circle (2cm);
			\draw (1.5,0.7) node {$B$};
		\end{tikzpicture}
		\caption{\cref{eg: 4. case} for $R=2$}
	\end{figure}
	Our analysis focuses on the fixed point
	\begin{align*}
		\ubar=\paren{0,R}-\frac{\lambda}{1-\lambda}\paren{0,R-1}
	\end{align*}
	of $\raar$ on $A$ and $B$. Note that it is enough to consider 
	$\ubar$ to get the analysis for any other fixed point due to the symmetry of
	the problem instance.
	
	Unfortunately, we were unable to verify the technical assumption \eqref{ass:3} in \cref{thm:conv raar}.
	
	Nevertheless, this example is subtransversal. The modulus of subtransversality $\kappa $ is
	bounded as follows
	\begin{align*}
		\kappa^2>\frac{2   R^{2}}{R^{4} - 2   R^{3} + 2   R^{2} - 2   R + 1}.
	\end{align*}
\end{eg}
\begin{eg}[tangential circles]\label{eg: 5. case}
	\cref{eg: 5. case} consists of 2 circles touching at a single point. Let $R>0$ and define
	\begin{align*}
		A&\equiv \set{x\in \Rbb^2 }{\norm{x}=1}\\
		B&\equiv \set{x\in \Rbb^2 }{ \norm{x-(R+1,0)}=R}.
	\end{align*}
	\begin{figure}[ht]
		\centering
		\begin{tikzpicture}
			\draw[thick,->] (-1.5,0) -- (3.5,0);
			\draw[thick,->] (0,-1.5) -- (0,1.5);
			\draw (0,0) circle (1cm);
			\draw (0.7,0.3) node {$A$};
			\draw (2,0) circle (1cm);
			\draw (2.7,0.3) node {$B$};
		\end{tikzpicture}
		\caption{\cref{eg: 5. case} for $R=1$}
	\end{figure}	
	Our convergence analysis focuses on the only point in the intersection of those 
	two sets, namely
	\begin{align*}
		\ubar=\paren{1,0}.
	\end{align*}	
	The following statements regarding the assumptions made in
	\cref{thm:conv raar} are easily verified either by hand or with the help of symbolic
	computation.
	\begin{enumerate}[(i)]
		\item $S_0\equiv \klam{\ubar}\in\Fix \raar $.
		\item The difference vector is unique as well and given by $\zetabar=\paren{(0,0),(0,0),(0,0),(0,0)}$.
		\item The sets $A$ and $B$ satisfy the assumptions of \cref{thm: singlevalued at fixed points}
			at $\ubar$ with neighborhoods $U_1$ and $U_2$ being open balls around $\ubar$, that is
			$\Ball_\delta(\ubar)$, for $\delta\in (0,1)$.
		\item The technical condition \eqref{ass:3} in \cref{thm:conv raar} is satisfied with
			\begin{align*}
				\sigma^2=\frac{1}{2   \lambda^{2}}
			\end{align*}
			by \cref{propn:technical ass for consistent}.
		\item However, this example is not subtransversal when examining it in $\Rbb^2$. Checking equivalently 
			linear regularity, which is fine since we are looking at a point in the intersection of the 
			two sets (see \cref{prop:lin reg}), yield a value of 
			\begin{align}\label{eq:value of lin reg}
				\frac{2   R}{{\left(R + 1\right)} b}
			\end{align}
			where we parametrized a point in the neighborhood of $\ubar$ intersected with $A$ as
			\begin{align*}
				(\sqrt{1-b^2},b), b\in [1,-1].
			\end{align*}
			As $b\to 0$ (in other words, for  points close to $\ubar$) the ratio \eqref{eq:value of lin reg}
			tends to $\infty$. This
			implies that \cref{eg: 5. case} cannot be linearly regular at the 
			point $\ubar=(1,0)$ and thus is not subtransversal there.
	\end{enumerate}
	The assumptions of \cref{thm:conv raar} therefore are not satisfied.  In light of the necessity 
	of metric subregularity for linear convergence \cite[Theorem 2]{LukTebTha18}, we 
	conclude that $\raar$ {\em cannot} be linearly convergent in this case (though it might be 
	sublinearly convergent). 
\end{eg}
\begin{rem}
As shown in the examples above the constants involved 
for both subtransversality and the technical condition \eqref{ass:3} in \cref{thm:conv raar}
can be cumbersome although the actual problem might look relatively easy. 
We also see in  \cref{eg: 2. case}-\cref{eg: 4. case} that the presence of 
subtransversality  in the inconsistent case can come as a surprise.  Our inability 
to show the technical condition \eqref{ass:3} in \cref{eg: 4. case} indicates that 
this condition characterizes the regularity or nondegeneracy of the underlying model space
for the algorithm. 
Further investigation of this property is needed. 
\end{rem}

%
%

%
%

\end{document}